\documentclass[10pt]{stijl}     
  
\usepackage{amssymb} 
\usepackage{amsmath} 
\usepackage{stmaryrd}
\usepackage{euscript}    
\usepackage{latexsym} 
\usepackage{mathrsfs} 
\usepackage{float} 
\usepackage[dvipsnames]{xcolor}
\usepackage{enumitem}
\usepackage[normalem]{ulem}
\usepackage{proof}
\usepackage{bussproofs} \EnableBpAbbreviations
\usepackage{tikz} 
\usetikzlibrary{positioning,arrows,calc,shapes.geometric}
\tikzset{ modal/.style={>=stealth,shorten >=1pt,shorten <=1pt,auto,node distance=1.2cm, semithick}, world/.style={ }, point/.style={circle,draw,inner sep=0.5mm,fill=black}, reflexive above/.style={->,loop,looseness=7,in=120,out=60}, reflexive below/.style={->,loop,looseness=7,in=240,out=300}, reflexive left/.style={->,loop,looseness=7,in=150,out=210}, reflexive right/.style={->,loop,looseness=7,in=30,out=330}, itria/.style={draw,shape border uses incircle, isosceles triangle,shape border rotate=270,yshift=-1.45cm, minimum height = 18mm} }
\usepackage{ctable} 
\usepackage{listliketab}
\usepackage{url}
\usepackage[round]{natbib} 
\newcommand{\proofs}{\vdash}
\newcommand{\impl}{\rightarrow}
\newcommand{\Impl}{\Rightarrow}
\newcommand{\calM}{\mathcal{M}}
\newcommand{\calK}{\mathcal{K}}
\newcommand{\calL}{\mathcal{L}}
\newcommand{\calX}{\mathcal{X}}
\newcommand{\calH}{\mathcal{H}}
\usepackage{lscape}

\newcommand{\HA}{{\sf HA}}

\newcommand{\IPC}{{\sf IPC}}

\newcommand{\K}{{\sf K}}

\newcommand{\GL}{{\sf GL}}
\newcommand{\iK}{{\sf iK}_\Box}

\newcommand{\iKD}{{\sf iKD}_\Box}

\newcommand{\iGL}{{\sf iGL}_\Box}
 
\newcommand{\iSL}{{\sf iSL}_\Box}

\newcommand{\LKGL}{{\sf G3GL}}
\newcommand{\LJ}{{\sf G3ip}}
\newcommand{\LJm}{{\sf G3im}}
\newcommand{\LJX}{{\sf G3iX}_\Box}

\newcommand{\LJK}{{\sf G3iK}_\Box}

\newcommand{\LJKD}{{\sf G3iKD}_\Box}

\newcommand{\LJGL}{{\sf G3iGL}_\Box}
\newcommand{\LJGLC}{{\sf G3iGLC}_\Box}
\newcommand{\LJSL}{{{\sf G3iSL}_\Box}}
\newcommand{\LJSLvier}{{{\sf G3iSL}^{4}_\Box}}
\newcommand{\LJSLa}{{{\sf G3iSL}^{a}_\Box}}
\newcommand{\DY}{{\sf G4ip}}
\newcommand{\DYm}{{\sf G4im}}

\newcommand{\DYK}{{\sf G4iK}_\Box}

\newcommand{\DYKD}{{\sf G4iKD}_\Box}

\newcommand{\DYGL}{{\sf G4iGL}_\Box}

\newcommand{\DYSL}{{{\sf G4iSL}_\Box}}
\newcommand{\DYSLprime}{{{\sf G4iSL}'_\Box}}
\newcommand{\DYSLa}{{{\sf G4iSL}^a_\Box}}

\newcommand{\G}{{\sf G}}

\newcommand{\iHSL}{{\sf iHSL}_\Box}

\newcommand{\iHGLC}{{\sf iHGLC}_\Box}

\newcommand{\rsch}{{\EuScript R}}

\newcommand{\lgc}{{\sf L}}
\newcommand{\lang}{\ensuremath {{\EuScript L}}}
\newcommand{\defn}{\equiv _{\mbox{\em \tiny df}}} 
\newcommand{\af}{\vdash}

\newcommand{\leftimprule}{\imp}

\newcommand{\imp}{\rightarrow}

\newcommand{\en}{\wedge} 
\newcommand{\of}{\vee}

\newcommand{\bx}{\raisebox{-.3mm}{$\Box$}}
\newcommand{\dbx}{\raisebox{.13mm}{$\boxdot$}}

\newcommand{\bof}{\bigvee}
\newcommand{\ben}{\bigwedge}
\newcommand{\seq}{\Rightarrow}
\newcommand{\sml}{\ll}
\newcommand{\cdwlord}{\prec}

\newcommand{\ord}{\sqsubset}

\newcommand{\cd}{{\it cd}}
\newcommand{\cw}{{\it cw}}
\newcommand{\cl}{{\it cl}}
\newcommand{\cdwl}{{\it dwl}}

\newcommand{\De}{\Delta}
\newcommand{\Ga}{\Gamma}

\newcommand{\Sig}{\Sigma}

\newcommand{\gam}{\gamma}
\renewcommand{\phi}{\varphi}
\newcommand{\varchi}{\raisebox{.4ex}{\mbox{$\chi$}}}

\newcommand{\calb}{{\EuScript B}}
\newcommand{\cald}{{\EuScript D}}

\newtheorem{Theor}{Theorem}
\newenvironment{theorem}{\begin{Theor}\em }{\end{Theor}}
\newtheorem{Lemma}{Lemma}
\newenvironment{lemma}{\begin{Lemma}\em }{\end{Lemma}}
\newtheorem{Coro}{Corollary}
\newenvironment{corollary}{\begin{Coro}\em }{\end{Coro}}
\newtheorem{Remark}{Remark}
\newenvironment{remark}{\begin{Remark}\em }{\end{Remark}}
\newtheorem{Claim}{Claim} 

\newtheorem{defin}{Definition}

\newtheorem{exam}{Example}

\newenvironment{proof}{{\bf Proof}}{\hfill $\slot$}
\newcommand{\slot}{\hfill \mbox{$\dashv$}}

\numberwithin{figure}{section}

\begin{document}
\title{Proof Theory for Intuitionistic Strong L\"ob Logic} 
\vskip5pt 
\author{
Iris van der Giessen\footnote{At the time of this version: Utrecht University, the Netherlands. At the time of arXiv submission (March 2023): University of Birmingham, United Kingdom, i.vandergiessen@bham.ac.uk.}\ \  \& Rosalie Iemhoff
\footnote{Utrecht University, the Netherlands, r.iemhoff@uu.nl. Support by the Netherlands Organisation for Scientific Research under grant 639.073.807 is gratefully acknowledged.} }
    
\maketitle

\begin{abstract}
 \noindent  
This paper introduces two sequent calculi for intuitionistic strong L\"ob logic $\iSL$: a terminating sequent calculus $\DYSL$ based on the terminating sequent calculus $\DY$ for intuitionistic propositional logic $\IPC$ and an extension $\LJSL$ of the standard cut-free sequent calculus $\LJ$ without structural rules for $\IPC$.  One of the main results is a syntactic proof of the cut-elimination theorem for $\LJSL$. In addition, equivalences between the sequent calculi and Hilbert systems for $\iSL$ are established. It is known from the literature that $\iSL$ is complete with respect to the class of intuitionistic modal Kripke models in which the modal relation is transitive, conversely well-founded and a subset of the intuitionistic relation. Here a constructive proof of this fact is obtained by using a countermodel construction based on a variant of $\DYSL$. The paper thus contains two proofs of cut-elimination, a semantic and a syntactic proof.
\end{abstract}



\section{Introduction}
In recent years there has appeared quite some research on nonclassical modal logic, research that stems from various motives.  
Nonclassical modal logics are, for example, related to semantical notions that appear in category theory and algebraic logic and can be used in these settings \citep{litak14}. Then there is the area of constructive modal logic, where the aim is to develop and investigate modal logics that have some form of computational content, such as a Curry-Howard interpretation or a realizability semantics \citep{paivaetal04}. Related but from a different perspective is research in philosophical logic that aims to give an account of modalities in an intuitionistic setting. One of the early approaches in this direction is \citep{simpson94}. Finally, there is research driven by the desire to understand in how far results and methods from classical modal logic are preserved when classical logic is replaced by a nonclassical logic, and in how far salient properties from a nonclassical logic are preserved when modalities are added. This intrinsically motivated line of research, which was mentioned in the literature for one of the first times in \citep{fischerservi77}, is the one that this paper mainly belongs to. Another motive is given at the end of the introduction. 

This paper is the third in a line of papers \citep{iemhoff18,giessen&iemhoff2019} in which terminating sequent calculi are developed for intuitionistic modal logics; for $\iK$, $\iKD$, $\iGL$ in the two mentioned papers, and for $\iSL$  in this paper. The reason for developing terminating sequent calculi for logics is that they are a useful tool in the investigation of a logic, in particular in establishing that a logic has uniform interpolation. In \citep{iemhoff16} the property of uniform interpolation was used to obtain general results on the existence of classical modal logics, and in \citep{iemhoff17} a similar approach was developed for intermediate logics as well as a small class of intuitionistic modal logics, thereby proving that $\iK$ and $\iKD$ have uniform interpolation. The results on classical modal logics and intermediate logics were generalized and extended in \citep{jalali&tabatabai2018}, that contains applications  to substructural logics, among other things. In all these papers, results on uniform interpolation heavily rely on the existence of terminating sequent calculi for the logics in question, as the interpolants are defined on the bases of such calculi. Recently, Litak and Visser have obtained, but not yet published, proofs of interpolation and uniform interpolation for $\iSL$, based on their work on the Lewis arrow in the context of intuitionistic logic \citep{litak&visser18}. Here we provide a syntactic proof of interpolation, but we currently do not know of a proof-theoretic proof of uniform interpolation for $\iSL$.

In this paper we develop sequent calculi for $\iSL$, the logic that consists of $\iK$ extended by the Strong L\"ob Principle 
\[
 {\rm (SL)} \ \ \  (\bx\phi \imp \phi) \imp \phi.
\]
Here $\iK$ is an intuitionistic version of the modal logic $\K$ with only the box modality. $\iSL$ is equivalent to the logic $\iK$ extended by the L\"ob Principle (left) and the Completeness Principle (right):
\[
  {\rm (GL)} \ \ \  \bx(\bx\phi \imp \phi) \imp \bx\phi \ \ \ \ \ \ \ 
  {\rm (CP)} \ \ \  \phi \imp \bx\phi.
\]
The logic $\iK$ extended only by {\rm GL} is denoted $\iGL$. 
In \citep{visser&zoethout18} it was shown that $\iGL + {\rm CP}$, and thus $\iSL$, is the provability logic of 
$\HA + (\phi \imp \bx_{(\HA^s)^*}\phi)$, where $\HA^s$ denotes an axiomatization of $\HA$ with only small axioms\footnote{See the paper for a definition.} and $(\HA^s)^*$ is the unique theory such that provably in $\HA$: $(\HA^s)^*=\HA^s + (\phi \imp \bx_{(\HA^s)^*}\phi)$. Thus $(\HA^s)^*$ is a nonclassical theory that proves its own completeness. (CP) naturally is part of the axiomatization of the provability logic of $\HA$ with respect to $\Sigma_1$-provability given in \citep{ardeshir&mojtahedi18}. 

Note that both $\iGL$ and $\iSL$ do not contain a diamond operator. From the provability point of view, $\Diamond$ can be interpreted as consistency, the dual of provability, and is thus defined in terms of negation and $\Box$. In the literature on intuitionistic modal logic it is standard to consider $\Diamond$ independently from $\Box$, but in this setting it is not clear what the natural interpetation of that operator should be.

In this paper we present two calculi for $\iSL$: $\LJSL$ and $\DYSL$. We show that $\LJSL$ has cut-elimination (Section~\ref{seccutelimination}) and that $\LJSL$ is equivalent to $\DYSL$ (Section~\ref{secequivalence}), thus proving that the latter has cut-elimination as well. In the case of $\iGL$ and $\iSL$ the syntactic cut-elimination proof is complicated, because the fact that already for classical $\GL$ establishing that the standard nonterminating calculus $\LKGL$ has cut-elimination is highly nontrivial \citep{gore&ramanayake2012,valentini83}. The cut-elimination proof in this paper uses the key idea of these two papers, introducing a new measure on proofs called the width. Section~\ref{seclogics} contains the preliminaries, including a discussion on the width of proofs. We furthermore show, using $\LJSL$, that $\iSL$ has interpolation (Section~\ref{secinterpolation}). 

The purpose of developing $\DYSL$ is that proof search in this calculus is strongly terminating (Section~\ref{sectermination}). This means that all possible applications of the rules yield a finite proof search tree. Decidability of the logic immediately follows.  Developing terminating sequent calculi for intuitionistic modal logics is in general harder than for their classical counterparts, as the standard calculi for intuitionistic propositional logic $\IPC$ are already nonterminating. However, a terminating sequent calculus $\DY$ for $\IPC$ has been developed in \citep{dyckhoff92} and this calculus has been used to obtain the terminating sequent calculi for $\iK$, $\iKD$, $\iGL$ mentioned above, and it is the basis for the terminating sequent calculus for~$\iSL$ introduced in this paper. These calculi are not extensions of $\DY$ by the standard modal sequent rules alone, because the intricate rules for left implication in $\DY$ have to be extended as well to cover the modal case. The order used to show strong termination is based on an order by \cite{bilkova06} used in her syntactic proof of uniform interpolation for~$\GL$. 

All these results are constructive and proved syntactically, and are closely related to results established for $\iGL$ in \citep{giessen&iemhoff2019}. New in the paper is a semantic treatment of the cut-elimination theorem as well (Section~\ref{Kripke iSL}). It immediately follows from the completeness proof via the Sch\"utte-Takeuti-Tait countermodel construction, which is based on $\DYSL$. It is shown that nonderivability in $\DYSL$ implies the existence of an~$\iSL$ countermodel. Thus giving a constructive proof of the completeness (already known from the work of \cite{ardeshir&mojtahedi18}) of $\iSL$ with respect to Kripke models that are transitive, conversely well-founded and in which the modal relation is a subset of the intuitionistic relation. 

The role of intuitionistic modal logic in philosophy lies in the fact that intuitionistic modal logics are formalizations of the logical properties of modalities considered from an intuitionistic point of view, modalities such as necessity or obligation or as in this paper, provability. Although the notion of proof is central in constructive reasoning, the study of formal intuitionistic provability is a highly nontrivial matter. Whereas it has been known for a long time that the axiomatization of classical provability in theories such as Peano Arithmetic is~$\GL$ \citep{solovay76}, it is unknown till this day what the axiomatization of the provability logic of $\HA$ is, although it is known that it is a proper extension of $\iGL$.

This makes $\iSL$ special, being one of the few known provability logics of arithmetical theories based on intuitionistic logic. Investigating $\iSL$ may further the understanding of formal constructive provability, and the proof theory developed in this paper hopefully furthers the investigation of $\iSL$.

\section{Logics and calculi}
 \label{seclogics}
The language $\lang$ for propositional intuitionistic modal logic considered here contains a {\em constant} $\bot$, {\em propositional variables} or {\em atoms} $p,q,r,\dots$, the {\em modal operator} $\bx$ and the {\em connectives} $\neg,\en,\of,\imp$, where $\neg\phi$ is defined as $(\phi \imp\bot)$. $\bot$ is by definition not an atom and $\Diamond$ does not occur in our language.  
For multisets $\Ga$ and $\Pi$, we denote by $\Ga \cup \Pi$ the multiset that contains only formulas $\phi$ that belong to $\Ga$ or $\Pi$ and the number of occurrences of $\phi$ in $\Ga \cup \Pi$ is the sum of the occurrences of $\phi$ in $\Ga$ and in $\Pi$.  We consider (single-conclusion) sequents, which are expressions $(\Ga \seq \De)$, where $\Ga$ and $\De$ are finite multisets of formulas and $\De$ contains at most one formula, and which are interpreted as $I(\Ga \seq \De) = (\ben\Ga \imp \bof\De)$. We denote finite multisets by $\Ga,\Pi,\De,\Sig$. In a sequent, $\Ga,\Pi$ is short for $\Ga\cup \Pi$, and $\Ga,\phi^n,\Pi$ denotes the union of $\Ga$, $\Pi$ and the multiset that consists of $n$ copies of $\phi $.   Furthermore ($a$ for antecedent, $s$ for succedent):  
\[
 (\Ga \seq \De)^a \defn \Ga \ \ \ \ (\Ga \seq \De)^s \defn \De \ \ \ \ 
 \bx\Ga \defn \{ \bx\phi \mid \phi \in \Ga \} \ \ \ \ \dbx\Ga \defn \Ga \cup\bx\Ga.
\]
$\dbx\phi$ is short for $\phi\en\bx\phi$, but if the expression occurs as an element of a sequent it stands for $\phi,\bx\phi$, meaning that $(\Ga,\dbx\phi\seq\De)$ should be read as  $(\Ga,\bx\phi, \phi\seq\De)$. For $\iSL$, $\boxdot \phi \leftrightarrow \phi$, because of the Completeness Principle $\phi \impl \Box \phi$. 

Given a sequent calculus \G\ and a sequent $S$, $\af_\G S$ denotes that $S$ is derivable in \G. The logic $\lgc_\G$ {\em corresponding} to \G\ is defined as  
\[
 \af_{\lgc_\G} \phi\ \defn \ \af_\G (\ \seq \phi). 
\]
The two main sequent calculi in this paper are  $\LJSL$ and $\DYSL$.

The {\em degree} of a formula $\phi$ is inductively defined by $d(\bot)=0$, $d(p)=1$, $d(\bx \phi)=d(\phi)+1$, and $d(\phi\circ \psi)= d(\phi)+d(\psi)+1$ for $\circ \in\{\en,\of,\imp\}$.
We use an order on sequents based on the {\em weight function} $w(\cdot)$ on formulas from \citep{dyckhoff92}, which is inductively defined as: the weight of an atom and the constant $\bot$ is 1, $w(\bx \phi)=w(\phi)+1$, and $w(\phi \circ \psi) = w(\phi)+w(\psi)+i$, where $i=1$ in case $\circ \in \{\of,\imp\}$ and $i=2$ if $\circ$ equals $\en$. 
We use the following ordering on sequents: $S_0 \sml S_1$ if and only if 
$S_0^a\cup S_0^s \sml S_1^a\cup S_1^s$, where $\sml$ is the order on multisets determined by weight as in \citep{dershowitz&manna79} (where they in fact define $\gg$): for multisets $\Ga,\De$ we have $\De \sml \Ga$ if $\De$ is the result of replacing one or more formulas in $\Ga$ by zero or more formulas of lower weight. 

For the termination of $\DYSL$ we need another order on sequents as well. Let $b(S)$ be the number of different boxed formulas in $S^a$, considered as a set. Given a number $c$, define the orderings $\prec^c$ and $\ord^c$ by: $S_1\prec^c S_2$ if and only if $c-b(S_1)< c-b(S_2)$, and $S_1 \ord^c S_2$ if and only if either $S_1 \prec^c S_2$ or $b(S_1)=b(S_2)$ and $S_1 \sml S_2$. As we will see, this order is used only in cases where $c \geq b(S)$.

\subsection{Intuitionistic modal logic}
 \label{secclassical}
One of the standard calculi without structural rules for $\IPC$ is $\LJ$, given in Figure~\ref{figgthm}, which is the propositional part of the calculus {\sf G3i} from \citep{troelstra&schwichtenberg96}. 
The calculus \DY\ in Figure~\ref{figdy} is the terminating sequent calculus for $\IPC$ independently discovered by \cite{dyckhoff92} and \cite{hudelmaier88,hudelmaier92,hudelmaier93}, where we have renamed some of the rules. Recall that in this paper sequents are assumed to have at most one formula in the succedent, which is a slight but inessential deviation from the calculus by Dyckhoff and Hudelmaier, where sequents are assumed to have exactly one formula at the right. The calculus is terminating with respect to the above well-ordering $\sml$ on sequents. The calculus has no structural rules, but they are admissible in it, as is shown in Section~\ref{secequivalence}.

\begin{figure}[!b]
 \centering
\[\small 
 \begin{array}{ll}
 \deduce{\Ga,p \seq p}{} \ \ \ \text{{\it At} \ \ ($p$ an atom)}  & 
  \deduce{\Ga,\bot\seq \De}{} \ \ \ L\bot \\
 \\ 
 \infer[R\en]{\Ga \seq \phi \en \psi}{\Ga\seq \phi & \Ga \seq \psi} & 
  \infer[L\en]{\Ga, \phi\en \psi \seq \De}{\Ga, \phi, \psi \seq \De} \\
 \\
 \infer[R\!\of \ (i=0,1)]{\Ga \seq \phi_0 \of \phi_1}{\Ga \seq \phi_i} & 
  \infer[L\of]{\Ga,\phi\of \psi\seq \De}{\Ga, \phi \seq \De & \Ga,\psi \seq \De} \\
  \\
 \infer[R\!\imp]{\Ga \seq \phi \imp \psi}{\Ga,\phi \seq \psi} & 
 \infer[L\!\imp]{\Ga,\phi\imp\psi\seq \De}{\Ga,\phi\imp\psi\seq \phi & \Ga,\psi\seq \De}
 \end{array}
\] 
\caption{The Gentzen calculus $\LJ$}
 \label{figgthm}
\end{figure}

\begin{figure}[!b]
 \centering
\[\small 
 \begin{array}{ll}
 \deduce{\Ga,p \seq p}{} \ \ \ \text{{\it At} \ \ ($p$ an atom)}  & 
  \deduce{\Ga,\bot\seq \De}{} \ \ \ L\bot \\
 \\ 
 \infer[R\en]{\Ga \seq \phi \en \psi}{\Ga\seq \phi & \Ga \seq \psi} & 
  \infer[L\en]{\Ga, \phi\en \psi \seq \De}{\Ga, \phi, \psi \seq \De} \\
 \\
 \infer[R\!\of \ (i=0,1)]{\Ga \seq \phi_0 \of \phi_1}{\Ga \seq \phi_i} & 
  \infer[L\of]{\Ga,\phi\of \psi\seq \De}{\Ga, \phi \seq \De & \Ga,\psi \seq \De} \\
  \\
 \infer[R\!\imp]{\Ga \seq \phi \imp \psi}{\Ga,\phi \seq \psi} & 
 \infer[Lp\!\imp\text{ ($p$ an atom)}]{\Ga, p,p \imp \phi \seq \De}{\Ga,p,\phi \seq \De}\\
 \\
 \infer[L\en\!\imp]{\Ga, \phi\en\psi \imp \gamma \seq \De}{\Ga,\phi\imp (\psi\imp\gamma)\seq\De} & 
 \infer[L\of\!\imp]{\Ga,\phi \of \psi \imp \gamma \seq \De}{
  \Ga,\phi \imp \gamma, \psi \imp \gamma \seq \De}\\ 
 \\ 
 \infer[L\imp\!\imp]{\Ga, (\phi\imp \psi) \imp \gamma \seq \De}{
  \Ga, \psi\imp \gamma \seq \phi \imp \psi & \gamma,\Ga \seq \De} &  
\end{array}
\] 
\caption{The Gentzen calculus $\DY$}
 \label{figdy} 
\end{figure}

As one can see, $\DY$ is the result of replacing the single left implication rule $L\!\imp$ in $\LJ$ by the four left implication rules of $\DY$, each corresponding to the outermost connective of the antecedent of the principal implication. When one analyses the nontermination of $\LJ$, one sees that in the standard ordering on sequents in every rule every premise is of lower complexity than its conclusion, except in the left implication rule, where  the principal implication occurs in the left premise. This explains why the only difference between $\DY$ and $\LJ$ is the left implication rule(s). Since in this paper the propositional language is extended by the modal operator $\bx$, we use five instead of four replacing implication rules, one extra for the case that the the antecedent of the principal implication is boxed.

The calculi considered in this paper are given in Figure~\ref{figlogics}. They are extensions of either~\LJ\ or~\DY\ by the modal rules and the additional left implication rules given in Figure~\ref{figmodal}. In the rules of $\LJX$, the {\em principal} formula of an inference are defined as usual for the connectives \citep{buss98}. For ${\rm X} \in \{{\rm GL,SL}\}$, the principal formulas are $\bx\phi$ and all formulas in $\bx\Ga$ for $\rsch_{\rm X}$, and for $\leftimprule_{\rm X}$ the principal formulas are $\bx\phi \imp \psi$ and the formulas in~$\bx\Ga$. The {\em diagonal} formula of $\rsch_{\rm GL}$ and $\rsch_{\rm SL}$ is $\bx\phi$. In an application of $\rsch_{\rm SL}$, the formulas in $\bx\Sig$ are said to be {\em introduced} in that inference step. Note that in $\LJSL$, $\Pi$ ranges over multisets that do not contain boxed formulas. In case of~$\iK$ and~$\iKD$ in \citep{iemhoff18} there is no requirement on $\Pi$ in the modal rules of the calculi~$\LJK$ and~$\LJKD$. The reason for the restriction on $\Pi$ and the presence of~$\bx\Sig$ in~$\LJSL$ lies in the proof of the cut-elimination theorem. We indicate explicitly in that proof (of Theorem~\ref{thmcutadm}) where these constraints are used.  

The notions {\em ancestor} is defined as usual, see \citep{buss98}. For example in rule $\rsch_{\rm SL}$, formula occurrences in $\Box \Sigma$ do not have an ancestor, and for formula occurrence $\Box B \in \Box \Gamma$ in the conclusion, both formula occurrences $\Box B,B \in \boxdot \Gamma$ are ancestors of it, and $\Box \phi$ and $\phi$ in the premise are ancestors of $\Box \phi$ in the conclusion. A {\em strict ancestor} of a formula $\phi$ is defined to be an ancestor of $\phi$ that as a formula is equal to $\phi$. In the rest of the paper we often omit these words. When in a sequent $S$ on a branch there occurs a formula $\phi$, then if we speak of an occurrence of $\phi$ in a sequent higher (lower) than $S$ along that branch, then we mean a strict ancestor (descendent) of that occurrence of $\phi$ in $S$.

\begin{figure}[t]
\[ \begin{array}{ccc} 
\\ \small
\\
  \infer[\rsch_{\rm GL}]{\Sig,\bx\Ga \seq \bx\phi}{\dbx\Ga, \bx\phi \seq \phi} & 
  \infer[\rsch_{\rm GL}^4]{\Pi,\bx\Ga \seq \bx\phi}{\dbx\Ga, \bx\phi \seq \phi} & 
  \infer[\leftimprule_{\rm GL}]{\Pi,\bx\Ga,\bx\phi \imp \psi \seq \De}{
   \dbx \Ga,\bx\phi \seq \phi & \Pi,\bx\Ga,\psi \seq \De}
\\
\\
  \infer[\rsch_{\rm SL}]{\bx\Sig,\Pi,\bx\Ga \seq \bx\phi}{\Pi,\dbx\Ga,\bx\phi \seq \phi} & 
  \infer[\rsch_{\rm SL}^4]{\Pi,\bx\Ga \seq \bx\phi}{\Pi,\dbx\Ga,\bx\phi \seq \phi}  &
  \infer[\leftimprule_{\rm SL}]{\Pi,\bx\Ga, \Box \phi \imp \psi \seq \De}{
   \Pi,\dbx \Ga,\bx\phi, \Box \phi \impl \psi \seq \phi & \Pi,\bx\Ga,\psi \seq \De}

\end{array}
\]
\caption{The modal rules. $\Pi$ ranges over multisets that do not contain boxed formulas.}  
 \label{figmodal}
 \end{figure}
\begin{figure}[t]
\[
 \begin{array}{lllll} 
\text{name} & \text{calculus} & & \text{name} & \text{calculus} 
\\
\LJGL\  & \LJm +\rsch_{\rm GL} & & \DYGL\ & \DYm + \leftimprule_{\rm GL} +\ \rsch_{\rm GL}^4 \\
\LJSL\ & \LJm +\rsch_{\rm SL} & & \DYSL\ & \DYm + \leftimprule_{\rm SL}  +\ \rsch_{\rm SL}^4 \\
 \end{array}
\]
\caption{$\LJm$ and $\DYm$ stand for the logic $\LJ$ and $\DY$, respectively,  but then formulated for the modal language of this paper.} 
\label{figlogics}
\end{figure}

The {\em height} of a derivation is the length of its longest branch, where branches consisting of one node are considered to have height 1. The height of a sequent in a derivation is the height of its subderivation. The {\em level} of a sequent is the sum of the heights of its premises and the level of an inference is the level of its conclusion.\footnote{In \citep{giessen&iemhoff2019} the level of an inference was called its {\em height}, following \citep{gore&ramanayake2012}. Here we use {\em level} instead, the terminology from \citep{troelstra&schwichtenberg96}} If $\af$ stands for derivability in a given calculus, then we write $\af^d S$ if $S$ has a proof of height at most $d$ in that calculus.

\subsection{Structural rules}
In the cut-elimination proof we need the usual lemma on the height-preserving admissibility of weakening, contraction and inversion (Lemma~\ref{lemstructural}), but for $\LJSL$ we also need a strengthening of the closure under weakening (Lemma~\ref{lemweakening}), for which we first have to introduce two transformations on proofs. Since in this paper all weakening (contraction) is weakening (contraction) at the left, we omit these last words and just speak of weakening (contraction).

\begin{lemma}[Weakening, contraction and inversion]\label{lemstructural}\\ 
For $\af$ denoting $\af_{\LJSL}$ or $\af_{\DYSL}$, the following statements hold. 

\begin{description}[labelwidth=3.1cm]

\item[{\it Falsum rule}]
$\af^d \Ga\seq \bot$ implies $\af^d \Ga \seq \De$, and $\af^d \Ga\seq \emptyset$ implies $\af^d \Ga \seq \De$.

\item[{\it Weakening}] 
$\af^d \Ga\seq \De$ implies $\af^d \Ga,\phi \seq \De$. 

\item[{\it Contraction $\LJSL$}]
$\af^d_\LJSL \Ga,\phi,\phi  \seq \De$ implies $\af^d_\LJSL \Ga,\phi \seq \De$.

\item[{\it Invertible rules}] $R\en$, $L\en$, $L\of$, $R\!\imp$, $Lp\!\imp$, $L\en\!\imp$, $L\of\!\imp$ are height-preserving invertible. 

\item[{\it Inversion L$\imp$}]
$\af^d \Ga,\phi\imp\psi \seq \De$ implies $\af^d \Ga, \psi \seq \De$.

\item[{\it Inversion L$\imp\!\imp$}]
$\af^d \Gamma, (\phi \imp \psi) \imp \gamma \Impl \Delta$ implies $\af^d \Ga, \gamma \seq \De$.

\item[{\it Inversion $\imp_{{\rm SL}}$}]
$\af^d \Pi, \Box \Gamma, \Box \phi \imp \psi \Impl \Delta$ implies $\af^d \Pi, \Box \Gamma, \psi \seq \De$.

\item[{\it Extended axiom}]
$\af \Ga, \phi \seq \phi$ for all $\phi$.

\end{description}

\end{lemma}
\begin{proof}
The last statement of the list is proved by induction, on the degree of $\phi$ in the case of $\LJSL$ and on the weight of $\phi$ in the case of $\DYSL$.
The other statements are proved by induction on height $d$. 
{\it Weakening} is used in the proofs of {\it Inversion L$\imp$} and {\it Invertible Rules}, which are used in the proof of {\it Contraction}. 

The proofs of these properties are standard. For details, see page 66--67 in \citep{troelstra&schwichtenberg96}. Let us only present the proof for {\it Weakening} in $\LJSL$ and inversion of $\imp_{{\rm SL}}$. 

For {\it Weakening}, suppose $\af^dS$. If $d=1$ and $S$ is an instance of an axiom, the statement clearly holds. Suppose $d>1$. If the last inference in the proof of $S$ is not a modal rule, the proof that $\af^d S^a,\phi\seq S^s$ is as usual. Therefore suppose the last inference is $\rsch_{\rm SL}$. Thus $S$ is of the form $(\bx\Sig,\Pi,\bx\Ga \seq \bx\psi)$ and the premise is $(\Pi,\dbx\Ga,\bx\psi \seq \psi)$, where $\Pi$ does not contain boxed formulas. We have to show that $\af^d (\bx\Sig,\Pi,\bx\Ga,\phi \seq \bx\psi)$. If $\phi$ is boxed, this follows immediately from an application of $\rsch_{\rm SL}$ to the premise in which $\phi$ is introduced in the conclusion. In case $\phi$ is not a boxed formula, then by the induction hypothesis $\af^{d-1} (\Pi,\dbx\Ga,\phi,\bx\psi\seq \psi)$, and an application of $\rsch_{\rm SL}$ gives $\af^d (\bx\Sig,\Pi,\bx\Ga,\phi \seq \bx\psi)$. 

To prove inversion of $\imp_{{\rm SL}}$, we use the fact that {\it Weakening} holds in $\DYSL$ which is similarly shown as above. Suppose $\af^d (\Pi, \Box \Gamma, \Box \phi \impl \psi \Impl \Delta)$. Let us only present the interesting case with $d>1$ and the last rule applied is $\imp_{{\rm SL}}$ for some principal formula in $\Pi$, say $\Box \phi_1 \impl \psi_1$. So the sequent has the form $( \Pi', \Box \phi_1 \impl \psi_1, \Box \Gamma, \Box \phi \impl \psi \Impl \Delta)$ and the premises are $(\Pi', \Box \phi_1, \Box \phi_1 \impl \psi_1, \boxdot \Gamma, \Box \phi \impl \psi \Impl \phi_1)$ and $(\Pi', \psi_1, \Box \Gamma, \Box \phi \impl \psi \Impl \Delta)$. By induction hypothesis we have $\proofs^{d-1} (\Pi', \Box \phi_1, \Box \phi_1 \impl \psi_1, \boxdot \Gamma, \psi \Impl \phi_1)$ and $\proofs^{d-1} (\Pi', \psi_1, \Box \Gamma, \psi \Impl \Delta)$. If $\psi$ is not boxed, an application of $\imp_{{\rm SL}}$ gives the desired results. Otherwise, $\psi$ is of the form $\Box \psi'$ and we use {\it Weakening} to see that $\proofs^{d-1} (\Pi', \Box \phi_1, \Box \phi_1 \impl \psi_1, \boxdot \Gamma, \Box \psi', \psi' \Impl \phi_1)$, and again we apply $\imp_{{\rm SL}}$  to conclude $\af^d (\Pi, \Box \Gamma, \psi \Impl \Delta)$.
\end{proof}

\begin{remark}\label{remark height-pres contraction}
We are only certain about the height-preserving admissibility of contraction in
$\LJSL$. For $\DYSL$ a direct proof might be difficult and it is the question whether it would be height-preserving. In \citep{giessen&iemhoff2019} we claimed this result for $\DYGL$ in Lemma 4.1, but this statement might be wrong. However, all other results remain true because these are independent from this. Finally, note that the admissibility of contraction in $\DYSL$ follows once we have shown equivalence to $\LJSL$ (Corollary \ref{corconcutdy}).
\end{remark}

\subsection{Critical inferences}
Let $\cald$ be a derivation in $\LJSL$ with endsequent $(\Ga \seq \bx \phi)$. A {\em $\bx \phi$-critical inference over $(\Ga \seq \bx \phi)$} is an $\rsch_{\rm SL}$-inference $R$ in $\cald$ such that

\begin{enumerate}
\item $\bx \phi$ is principal in $R$; 
\item between $(\Ga \seq \bx \phi)$ and the conclusion of $R$ there is exactly one $\rsch_{\rm SL}$-inference in which~$\bx \phi$ is principal, and it is the diagonal formula of that inference. 
\end{enumerate}
The segment from the premise of $R$ till $(\Ga \seq \bx \phi)$ is a {\em $\bx \phi$-critical segment}. 
It is important to recall our convention  from Section~\ref{secclassical} about strict ancestors; in the definition of $\bx \phi$-critical inference, the $\bx \phi$ mentioned in 1.\ is required to be a strict ancestor of the $\bx \phi$ mentioned in 2.

\begin{remark}
 \label{remcritical} 
We are mostly interested in the case that $(\Ga \seq \bx\phi)$ is the conclusion of an $\rsch_{\rm SL}$-inference in a proof in $\LJSL$. In that case the $\bx\phi$-critical segment looks as follows, with $\Ga = \Box \Sigma_1,\Pi_1,\Box \Gamma_1$ and where $\Pi_1$ and $\Pi_2$ do not contain boxed formulas, and where $\Ga_1' \subseteq \Ga_1$:
\[\small 
 \infer[\rsch_{\rm SL}]{(S) \  \bx\Sig_1,\Pi_1,\bx\Ga_1\seq \bx\phi \ }{
  \deduce{\Pi_1,\dbx\Ga_1,\bx\phi\seq \phi}{\deduce{\calb}{
   \infer[R \text{ (application of }\rsch_{\rm SL})]{\bx\Sig_2,\Pi_2,\bx\Ga_1,\bx\Ga_2,\bx\phi \seq \bx\varchi}{
    \Pi_2,\dbx\Ga_1',\dbx\Ga_2,\dbx\phi,\bx\varchi \seq \varchi} }  } }
\]

It easily follows that there can be no applications of $\rsch_{\rm SL}$ in $\calb$, for at such an inference $\bx\phi$ would have disappeared in the sequents above it, as it is not principal in that inference by definition.\footnote{This is where we use the restriction that all formulas in the $\Pi$ of an application of $\rsch_{\rm SL}$ are not boxed.} As mentioned above, by definition the $\bx\phi$ in the top sequent is required to be a strict ancestor of the $\bx\phi$ in the bottom sequent. So only rules from $\LJm$ can be applied in the segment of $\calb$ indicated above, and therefore boxed formulas in
the antecedents of the sequents do not disappear in the backward direction of the proof tree. Therefore $\bx \Ga_1$ is still present in the sequent pictured above $\calb$.

Note that above $(\Gamma \seq \bx \phi)$ there may be more sequents than the ones in the $\bx\phi$-critical segment, as in the following example, where $\phi=(\phi'\imp \bx\chi)$ and $\phi'$ is not boxed, and the leftmost branch is a $\bx\phi$-critical segment: 
\[\small 
 \infer[\rsch_{\rm SL}]{\psi_1\of\psi_2\seq \bx\phi}{
  \infer[{\rm L}\of]{\psi_1\of\psi_2,\bx\phi\seq \phi}{
   \infer[{\rm R}\!\imp]{\psi_1,\bx\phi\seq \phi}{
    \infer[\rsch_{\rm SL}]{\psi_1,\bx\phi,\phi'\seq \bx\varchi}{
     \psi_1,\dbx\phi,\phi',\bx\varchi \seq \varchi} }   
  &
  \deduce{\psi_2,\bx\phi\seq \phi}{\cald'} 
  } } 
\] 

\end{remark}

\subsection{Weakening in critical segments}\label{subsecweak}
In the proof of cut-admissibility for $\LJSL$, we do not only need the closure under weakening as established in the previous lemma, but we also need to weaken $\bx\phi$-critical segments in a specific way. Namely, the situation may occur that we wish to weaken the sequents in such a segment in such a way that all inferences remain valid and the top sequent is weakened by $\varchi$ and the bottom sequent by $\bx\varchi$. As an example, consider a $\bx\bx\phi$-critical segment of minimal length: 
\[\small 
  \infer[\rsch_{\rm SL}]{\seq \bx\bx\phi}{
  \infer[\rsch_{\rm SL}]{\bx\bx\phi \seq \bx \phi}{
   \deduce{\dbx\bx\phi,\bx \phi \seq \phi}{} 
   }
  } 
\]

The following are two ways to achieve this, depending on whether $\varchi$ is boxed (left) or not boxed (right).
\[\small 
 \infer[\rsch_{\rm SL}]{\bx\varchi \seq \bx\bx\phi}{
  \infer[\rsch_{\rm SL}]{\bx\varchi,\varchi,\bx\bx\phi \seq \bx \phi}{
   \deduce{\bx\varchi,\varchi,\dbx\bx\phi,\bx \phi \seq \phi}{} 
  } 
 }
 \ \ \ \  
 \infer[\rsch_{\rm SL}]{\bx\varchi \seq \bx\bx\phi}{
  \infer[\rsch_{\rm SL}]{\bx\varchi,\varchi,\bx\bx\phi \seq \bx \phi}{
   \deduce{\bx\varchi,\varchi,\varchi,\dbx\bx\phi,\bx \phi \seq \phi}{} 
  } 
 }
\]

In general, if for a $\bx\phi$-critical segment we need to weaken its top sequent with at least $\varchi$ and its bottom sequent $S$ with $\bx\varchi$, then every sequent in $\cald$ has to be weakened at the left by $\bx\varchi,\varchi^n$, where $n$ is the number of applications of $\rsch_{\rm SL}$ below that sequent, except in the case that $\varchi$ is boxed, as in that case it can be introduced at every application of $\rsch_{\rm SL}$. Of course, in order to remain a valid derivation also sequents not in that segment but above $S$ have to be considered. In what follows the details behind this idea are spelled out. 

The {\em $\rsch_{\rm SL}$-grade} $g_{\cald}(S)$ of $S$ is defined as 
\[
 g_{\cald}(S) \defn \text{the number of applications of $\rsch_{\rm SL}$ below $S$ in $\cald$.} 
\]
We write $g(S)$ if $\cald$ is clear from the context.  
In this measure, if $S$ is the premise of an application of $\rsch_{\rm SL}$, 
then that application counts in $g(S)$. 
For example, on the following two branch segments, 
\[\small 
 \deduce{\vdots}{
  \infer[\rsch_{\rm SL}]{S_1}{
  \infer[\text{nonmodal rule}]{S_2}{
   \deduce{S_3}{
    \deduce{\vdots}{
    \infer[\rsch_{\rm SL}]{S_5}{
     \deduce{S_6}{\vdots} }  }  } 
  & & 
   \deduce{S_4}{
    \deduce{\vdots}{}  } 
  }
 } }
\]
if there is no application of $\rsch_{\rm SL}$ below $S_1$, then $g(S_1)=0$, $g(S_6)=2$, and $g(S_i)=1$ for all other $i$. 

Given a sequent $S=(\Ga\seq\De)$ and a derivation $\cald$, translations $(\cdot)^{\footnotesize \varchi}$ and $(\cdot)_{\footnotesize \varchi}$ on sequents in~$\cald$ are defined as follows 
(we suppress the dependence on $\cald$ in the notation, as it will always be clear from the context which derivation is meant): 
\[
 S^{\footnotesize \varchi} = (\Ga\seq \De)^{\footnotesize \varchi} \defn 
  \left\{ 
  \begin{array}{ll}
   \Ga \seq \De & \text{if $\varchi$ is boxed and $g(S)>0$} \\
   \Ga,\varchi \seq \De & \text{if $\varchi$ is not boxed or $g(S)=0$.} \\
  \end{array}
 \right.
\]
\[
 S_{\footnotesize \varchi} = (\Ga\seq \De)_{\footnotesize \varchi} \defn 
  \left\{ 
  \begin{array}{ll}
   \Ga,\bx\varchi \seq \De & \text{if $\varchi$ is boxed and $g(S)=0$} \\
   \Ga,\bx\varchi,\varchi \seq \De & \text{if $\varchi$ is boxed and $g(S)>0$} \\
   \Ga,\bx\varchi,\varchi^{g(S)}\seq \De & \text{if $\varchi$ is not boxed.} 
  \end{array}
 \right.
\]
$\cald_{\footnotesize \varchi}$ is the result of replacing each sequent $S$ in $\cald$ by 
$S_{\footnotesize \varchi}$, likewise for $\cald^{\footnotesize \varchi}$. 
Given a multiset of formulas $\{\varchi_1, \dots,\varchi_n\}$, define 
\[
 \cald^{\footnotesize \{\varchi_1, \dots,\varchi_n\}} \defn (\dots ((\cald^{\footnotesize \varchi_1})^{\footnotesize \varchi_2}) \dots )^{\footnotesize \varchi_n},
\]
and similarly for $\cald_{\footnotesize \{\varchi_1, \dots,\varchi_n\}}$.

\begin{lemma}[Strong weakening]
 \label{lemweakening} 
 \\
For any sequent $S=(\Ga \seq \De)$ and multiset of formulas $\Theta$ we have in $\LJSL$: 
\begin{enumerate}
\item If $\cald$ is a cutfree proof of $S$, then 
$\cald^{\footnotesize\Theta}$ is a cutfree proof of $(\Ga,\Theta\seq\De)$ of the same height as $\cald$. 

\item 
If $\cald$ is a cutfree proof of $S$, then 
$\cald_{\footnotesize\Theta}$ is a cutfree proof of $(\Ga,\Box\Theta\seq\De)$ of the same height as $\cald$. Moreover, given the following derivation $\cald$ that ends with $\rsch_{\rm SL}$ and has a $\bx\phi$-critical inference above its endsequent $S$:
\[\small 
 \infer[\rsch_{\rm SL}]{(S) \ \ \ \bx\Sig_1,\Pi_1,\bx\Ga_1\seq \bx\phi \ \ \ }{
  \deduce{\Pi_1,\dbx\Ga_1,\bx\phi\seq \phi}{\deduce{\calb}{
   \infer[\rsch_{\rm SL}]{\bx\Sig_2,\Pi_2,\bx\Ga_2,\bx\phi \seq \bx\varchi}{
    (S') \ \ \ \Pi_2,\dbx\Ga_2,\dbx\phi,\bx\varchi \seq \varchi \ \ \ } }  } }
\] 
Then in $\cald_{\footnotesize\varchi}$ this part of the proof becomes:
\[\small 
  \infer[\rsch_{\rm SL}]{\bx\Sig_1,\Pi_1,\bx\Ga_1,\bx\varchi\seq \bx\phi}{
   \deduce[\calb']{\Pi_1,\dbx\Ga_1,\bx\varchi, \varchi,\bx\phi \seq \phi}{
    \infer[\rsch_{\rm SL}]{\bx\Sig_{2},\Pi_{2},\bx\Ga_{2},\bx\varchi,\varchi,\bx\phi \seq \bx\varchi}{
    \Pi_{2},\dbx\Ga_{2},\dbx\phi,\bx\varchi,\varchi^k,\bx\varchi \seq \varchi} } } 
\]
where $k=1$ if $\varchi$ is a boxed formula and $k=g(S')=2$ otherwise. $\calb'$ is that part of  the branch segment in $\cald_{\footnotesize\varchi}$ that corresponds to $\calb$ in $\cald$. In $\cald_{\varchi}$, the topmost sequent is weakened at least by $\varchi$ and its bottom sequent by $\bx\varchi$.  
\end{enumerate}
\end{lemma}
\begin{proof}
The proof of 1.\ is straightforward. Without lose of generality, we show 2.\ for $\Theta = \{\chi \}$. It is clear that for every axiom $S$ in $\cald$, $S_{ \chi}$ is an axiom too. So all leafs of $\cald_{\footnotesize \varchi}$ are axioms. Therefore it suffices to show that all inference steps remain valid under translation $(\cdot)_{\varchi}$. Consider an application of a two-premise rule:
\[
 \infer[R]{S_1}{S_2 & S_3}
\]
If $R$ is an instance of a nonmodal rule $\rsch$, then $g(S_1)=g(S_2)=g(S_3)$. Thus sequents $S_i=(\Ga_i\seq \De_i)$ are in $(S_i)_{\footnotesize \varchi}$ all weakened with the same formula(s) on the left or remain all three  unchanged. Therefore the inference becomes, for $\Pi$ being either $\varnothing$, $\{\bx\varchi\}$, $\{\bx\varchi,\varchi\}$, $\{\bx\varchi,\varchi^{g(S_1)}\}$:  
\[
 \infer[R]{\Pi,\Ga_1 \seq \De_1}{\Pi,\Ga_2 \seq \De_2 & \Pi,\Ga_3 \seq \De_3}
\] 
This clearly is an instance of $\rsch$ as well. Single premise rules are treated in exactly the same way. For the case that $R$ is an instance of a modal rule, suppose it is of the form   
\[
 \infer[\rsch_{\rm SL}]{(S_1)\ \bx\Sig,\Pi,\bx\Ga \seq \bx\phi}{
  \deduce{(S_2)\ \Pi,\dbx\Ga,\bx\phi \seq \phi}{} } 
\]
Here $\Pi$ does not contain boxed formulas. We distinguish the case that $\varchi$ is a boxed formula and that it is not. In the first case, the inference becomes one of
\[
 \infer[\rsch_{\rm SL}]{\bx\Sig,\Pi,\bx\Ga,\bx\varchi,\varchi \seq \bx\phi}{
  \deduce{\Pi,\dbx\Ga,\bx\varchi,\varchi,\bx\phi \seq \phi}{} } 
 \ \ \ \ 
 \infer[\rsch_{\rm SL}]{\bx\Sig,\Pi,\bx\Ga,\bx\varchi\seq \bx\phi}{
  \deduce{\Pi,\dbx\Ga,\bx\varchi,\varchi,\bx\phi \seq \phi}{} } 
\]
depending on whether $g(S_1)>0$ or $g(S_1)=0$.
These are instances of $\rsch_{\rm SL}$, as the $\varchi$ in the conclusion of the leftmost case can be introduced because it is a boxed formula. If $\varchi$ is not a boxed formula, then because $g(S_2)=g(S_1)+1$, the inference becomes  
\[
 \infer[\rsch_{\rm SL}]{\bx\Sig,\Pi,\bx\Ga,\bx\varchi,\varchi^{g(S_1)} \seq \bx\phi}{
  \deduce{\Pi,\dbx\Ga,\bx\varchi,\varchi,\varchi^{g(S_1)},\bx\phi \seq \phi}{} } 
\]
Since $\varchi$ is not boxed, the formulas $\varchi^{g(S_1)}$ in the premise remain in the conclusion, and the inference is indeed valid. 

It is not hard to see that $\cald_{\footnotesize \varchi}$ is cutfree and has the same height as $\cald$. 
Finally, because for the endsequent $S$ we have $g(S)=0$, the endsequent of $\cald_{\footnotesize \varchi}$ is $(\Ga,\bx\varchi\seq \De)$. 
\end{proof}

\section{Cut-elimination}
 \label{seccutelimination}
In this section we show that the rule
\[
 \infer[{\it Cut}]{\Ga_1, \Ga_2 \seq \De}{\Ga_1 \seq \phi & \phi,\Ga_2 \seq \De}
\]
is admissible in $\LJSL$. This fact is used to prove, in Section~\ref{secequivalence}, the equivalence of this logic to its $\DY$ counterpart $\DYSL$. The proof of cut-elimination for $\LJSL$ is not straigthforward and uses a method similar to the one in \citep{giessen&iemhoff2019} for~$\LJGL$, which again is based on the method to prove cut-elimination for {\sf GL} from \citep{gore&ramanayake2012,valentini83}. The key idea behind these proofs is the use of a new measure on cuts, the width, which is defined as follows. 

Consider a topmost cut of the form
\[\small 
 \infer{\Ga_1,\Ga_2 \seq \psi}{
  \deduce[\cald_1]{\Ga_1 \seq \phi}{} & 
  \deduce[\cald_2]{\phi,\Ga_2 \seq \psi}{}
  }
\]
where $\cald_1$ and $\cald_2$ are cutfree derivations. The {\em width} of the cut equals $0$ if $\phi$ is not boxed. In case $\phi$ is boxed, the {\em width}  is the number of $\phi$-critical inferences over $\Ga_1 \seq \phi$.

The {\em cutdegree}, denoted \cd($\cald$), of $\cald$ is the maximal degree of the cutformulas in $\cald$. If a cut in~$\cald$ has degree \cd($\cald$), we call it a {\em maximal cut} and its cutformula a {\em maximal cutformula}. The {\em cutwidth}, \cw($\cald$), of $\cald$ is the maximal width among the maximal cuts in $\cald$. The {\em cutlevel\/}, \cl($\cald$), of $\cald$ is the maximal level among the maximal cuts in $\cald$ which have maximal width, i.e.\ which width is \cw($\cald$). We define
\[
 \cdwl(\cald) \defn (\cd(\cald),\cw(\cald),\cl(\cald)).
\] 
We consider these triples ordered by the lexicographic ordering, the {\em dwl-order}, and write $\cald' \cdwlord \cald$ whenever $\cdwl(\cald')$ comes before $\cdwl(\cald)$ in that ordering.

\begin{lemma}
 \label{lemnoncritical}
In $\LJSL$, if a cut of the form 
\begin{equation}\label{eqcritical}\small 
 \infer[{\it Cut}]{\bx\Sig_1,\Pi_1,\bx\Ga_1,\Pi_2,\bx\Ga_2\seq \psi}{
  \infer[\rsch_{\rm SL}]{\bx\Sig_1,\Pi_1,\bx\Ga_1 \seq \bx\phi}{\deduce[\cald]{\Pi_1,\dbx\Ga_1,\bx\phi\seq \phi}{} } & 
  \deduce[\cald']{\bx\phi,\Pi_2,\bx\Ga_2\seq \psi}{} }
\end{equation}
has width zero and $\cald$ is cutfree, then there is a cutfree derivation of $(\Pi_1,\dbx\Ga_1 \seq \phi)$. 
\end{lemma}
\begin{proof}
As the cut has width zero, in any inference in the left premise, any strict ancestor of the indicated $\bx\phi$ in the antecedent of the conclusion of $\cald$ is not principal. Therefore in any inference in $\cald$ (strict ancestors of) $\bx\phi$ occur either in both conclusion and premises or in none, or the inference has the following form
\[\small 
 \infer[\rsch_{\rm SL}]{\bx\phi,\bx\Sig,\Pi,\bx\Ga \seq \bx\psi}{\Pi,\dbx\Ga,\bx\psi \seq \psi}
\] 
All these inferences remain valid if the strict ancestors of the occurrence of $\bx\phi$ that we are considering are removed in the antecedents, if present. The same applies to axioms. Thus removing in the antecedents of the sequents in $\cald$ bottom-up the strict ancestors of $\bx\phi$ as long as they are present, results in a cutfree proof of $(\Pi_1,\dbx\Ga_1 \seq \phi)$.
\end{proof}

\begin{theorem}[Cut-elimination]
 \label{thmcutadm}
\\
The {\it Cut} Rule is admissible in $\LJSL$.
\end{theorem}
\begin{proof}
Let $\af$ denote $\af_{\LJSL}$. Following the corrected version \citep{troelstra98} of the cut elimination proof for $\LJ$ in \citep{troelstra&schwichtenberg96}, we successively eliminate cuts from the proof, always considering those cuts that have no cuts above them, the {\em topmost} cuts. For this it suffices to show that for cutfree proofs $\cald_1$ and $\cald_2$, the following proof $\cald$ can be transformed into a cutfree proof $\cald'$ of the same endsequent: 
\[\small 
 \infer[{\it Cut}]{\Ga_1, \Ga_2 \seq \De}{
  \deduce[\cald_1]{\Ga_1 \seq \phi}{} & 
  \deduce[\cald_2]{\phi,\Ga_2 \seq \De}{} 
 }
\]
This is proved by induction on the dwl-order $\cdwlord$ defined above. We use the fact that $\LJSL$ is closed under weakening and contraction implicitly at various places. 

There are three possibilities:

\begin{enumerate}
\item at least one of the premises is an axiom;
\item both premises are not axioms and the cutformula is not principal in at least one of the premises;
\item the cutformula is principal in both premises, which are not axioms.  
\end{enumerate}

1.\ As in \citep{troelstra&schwichtenberg96}, straightforward, by checking all possible cases: 
If $\cald_1$ is axiom $L\bot$, then we let $\cald'$ be the instance $S$ of that axiom. If $\cald_2$ is axiom $L\bot$ and $\bot$ is not the cutformula, then we let $\cald'$ be the instance $S$ of that axiom. If $\bot$ is the cutformula, then $(\Ga_1\seq\bot)$ has a cutfree derivation. Thus so has $(\Ga_1,\Ga_2\seq\De)$ by Lemma~\ref{lemstructural}, where we use the additional fact that in this proof no new cuts are introduced with respect to $\cald_1$. Therefore we let $\cald'$ be this proof. 

Assume both premises are not instances of $L\bot$. 
If~$\cald_1$ is axiom {\it At}, then $\phi$ is an atom. If~$\cald_2$ also is an instance of {\it At}, then  $S$ is an instance of {\it At}, so we let $\cald'$ be that instance. If~$\cald_2$ is not an axiom, then $\phi$ cannot be principal in its last inference, because it is an atom. We obtain $\cald'$ by {\em cutting at a lower level}, by which mean the following. Suppose the last inference of $\cald_2$ is $R\!\imp$:
\[
 \infer[{\it Cut}]{\Ga_1,\Ga_2\seq \psi_1\imp\psi_2}{
  \deduce{\Ga_1 \seq \phi}{\cald_1} 
  & 
  \infer{\Ga_2,\phi \seq \psi_1 \imp \psi_2}{
   \deduce{\Ga_2,\phi,\psi_1 \seq \psi_2}{\cald_2'} }
 }
\]
Consider the proof  $\cald_3$: 
\[  
 \infer[{\it Cut}]{\Ga_1,\Ga_2,\psi_1 \seq \psi_2}{
  \deduce{\Ga_1 \seq \phi}{\cald_1}
  & 
   \deduce{\Ga_2,\phi,\psi_1 \seq \psi_2}{\cald_2'} 
 }  
\]
Since $\cald_3 \cdwlord \cald$, there is a cutfree proof $\cald_3'$ of $(\Ga_1,\Ga_2,\psi_1 \seq \psi_2)$. Hence the following derivation is the desired cutfree proof $\cald'$:
\[ 
 \infer{\Ga_1,\Ga_2\seq \psi_1\imp\psi_2}{
 \deduce[\cald_3']{\Ga_1,\Ga_2,\psi_1 \seq \psi_2}{}
 } 
\]
The cases where the last inference of $\cald_2$ is another rule than $R\imp$ are treated in a similar way. 

The case that $\cald_1$ is not an axiom and $\cald_2$ is an instance of {\it At} remains. Here we also cut at a lower level, the proof is completely analogous to the case just treated.

2.\ First, the case that $\phi$ is not principal in the last inference of $\cald_1$. Thus the last inference in $\cald_1$ is one of the nonmodal rules $\rsch$ of $\LJSL$. In case $\rsch$ is L$\imp$, the last part of the proof looks as follows:
\[\small 
  \infer[{\it Cut}]{\Ga,\Ga_2, \psi \imp\psi' \seq \De}{
  \infer[\rsch]{\Ga,\psi\imp\psi'\seq \phi}{
   \deduce[\cald_1']{\Ga,\psi\imp\psi' \seq \psi}{} & 
   \deduce[\cald_1'']{\Ga,\psi' \seq \phi}{}} & 
  \deduce[\cald_2]{\Ga_2,\phi \seq \De}{} 
  }
\] 
where $\Ga_1 = \Ga,\psi\imp\psi'$. 
The cut can be pushed upwards in the following way (double lines suppress applications of weakening):  
\[\small 
  \infer[\rsch]{\Ga,\Ga_2, \psi\imp\psi'\seq \De}{
  \infer={\Ga,\Ga_2,\psi\imp \psi' \seq \psi }{
   \deduce[\cald_1']{\Ga,\psi\imp \psi' \seq \psi}{} } 
  & 
  \infer[{\it Cut}]{\Ga,\Ga_2,\psi' \seq \De}{
   \deduce[\cald_1'']{\Ga,\psi' \seq \phi}{} & \deduce[\cald_2]{\Ga_2,\phi \seq \De}{} } 
  }
\]
Thus we obtain a proof of $(\Ga_1,\Ga_2,\psi\imp\psi' \seq \De)$ with a cut of the same degree and at most the same width as the cut in $\cald$ but of lower level, and the induction hypothesis can be applied. 
Other nonmodal rules can be treated in a similar way. 

Second, the case that $\phi$ is not principal in the last inference of $\cald_2$. The nonmodal rules are handled as in the previous case. We treat the remaining case, where the lower part of $\cald$ takes one of the following forms:
\[\small 
  \infer{\Ga_1,\bx\Sig,\Pi,\bx\Ga \seq \bx\psi}{
  \deduce[\cald_1]{\Ga_1\seq \phi}{} &   
  \infer[\rsch_{\rm SL}]{\phi,\bx\Sig,\Pi,\bx\Ga \seq \bx \psi}{\deduce[\cald'_2]{\phi,\Pi,\dbx\Ga,\bx\psi \seq \psi}{} } }
  \ \ \ \ 
  \infer{\Ga_1,\bx\Sig,\Pi,\bx\Ga \seq \bx\psi}{
  \deduce[\cald_1]{\Ga_1\seq \phi}{} &   
  \infer[\rsch_{\rm SL}]{\phi,\bx\Sig,\Pi,\bx\Ga \seq \bx \psi}{\deduce[\cald'_2]{\Pi,\dbx\Ga,\bx\psi \seq \psi}{} } }
\] 
depending on whether $\phi$ is not a boxed formula (the leftmost case) or is a boxed formula (the rightmost case). 
Here $\Ga_2=\bx\Sig,\Pi,\bx\Ga$ and $\Pi$ contains no boxed formulas. 
In the rightmost case, consider the following cutfree proof:
\[\small 
  \infer[\rsch_{\rm SL}]{\bx\Sig,\Pi,\bx\Ga \seq \bx \psi}{
   \deduce[\cald'_2]{\Pi,\dbx\Ga,\bx\psi \seq \psi}{} 
  } 
\] 
Closure under weakening implies the existence of a cutfree proof of $(\Ga_1,\bx\Sig,\Pi,\bx\Ga \seq \bx \psi)$. 
In the leftmost case, let $\Ga',\Pi'$ be such that $\Ga_1 = \bx\Ga' \cup \Pi'$ and $\Pi'$ does not contain boxed formulas. The following is a proof of the same endsequent of the same degree and width but of lower level (double lines suppress applications of weakening):
\[\small 
 \infer[\rsch_{\rm SL}]{\Ga_1,\bx\Sig,\Pi,\bx\Ga \seq \bx\psi}{
 \infer=[{\it Weakening}]{\dbx\Ga',\Pi',\Pi,\dbx\Ga,\bx\psi \seq \psi}{
  \infer{\Ga_1,\Pi,\dbx\Ga,\bx\psi \seq \psi}{
   \deduce[\cald_1]{\Ga_1\seq \phi}{} & 
   \deduce[\cald'_2]{\phi,\Pi,\dbx\Ga,\bx\psi \seq \psi}{} } } }
\] 

3.\ The cutformula is principal in both premises, which are not axioms. We distinguish by cases according to the form of the cutformula, and treat implications and boxed formulas. 

If the cutformula is an implication, the last inference has the following form: 
\[\small 
 \infer[{\it Cut}]{\Ga_1,\Ga_2 \seq \De}{
  \infer[R\!\imp]{\Ga_1 \seq \phi\imp\psi}{\deduce[\cald_1']{\Ga_1,\phi \seq \psi}{} } & 
  \infer[L\!\imp]{\Ga_2,\phi\imp\psi \seq \De}{
   \deduce[\cald_2']{\Ga_2,\phi\imp\psi \seq \phi}{} & 
   \deduce[\cald_2'']{\Ga_2,\psi \seq \De}{} } 
  }
\]
This is replaced by the proof 
\[\small 
 \infer[\it Cut]{\Ga_1,\Ga_1,\Ga_2,\Ga_2 \seq \De}{
  \infer[\it Cut]{\Ga_1,\Ga_1,\Ga_2 \seq \psi}{
   \infer[\it Cut]{\Ga_1,\Ga_2 \seq \phi}{
   \infer{\Ga_1 \seq \phi\imp\psi}{\deduce[\cald_1']{\Ga_1,\phi \seq \psi}{} } & 
   \deduce[\cald_2']{\Ga_2,\phi\imp\psi \seq \phi}{} } & 
  \deduce[\cald_1']{\Ga_1,\phi \seq \psi}{} } & 
  \deduce[\cald_2'']{\Ga_2,\psi \seq \De}{}
 }
\]
in which each cut either is of lower degree or of the same degree and width but of lower level as the cut in $\cald$. 

If the cutformula is a boxed formula, then the proofs of the premises end with an application of $\rsch_{\rm SL}$:
\[\small 
 \infer[{\it Cut}_1]{\bx\Sig_1,\bx\Sig_2,\Pi_1,\Pi_2,\bx\Ga_1,\bx\Ga_2 \seq \bx\psi}{ 
  \infer[\rsch_{\rm SL}]{\bx\Sig_1,\Pi_1,\bx\Ga_1\seq \bx\phi}{
   \deduce[\cald_1']{\Pi_1,\dbx\Ga_1,\bx\phi \seq \phi}{} } & 
  \infer[\rsch_{\rm SL}]{\bx\Sig_2,\Pi_2,\bx\Ga_2,\bx\phi \seq\bx\psi}{
   \deduce[\cald_2']{\Pi_2,\dbx\Ga_2,\dbx\phi,\bx\psi\seq\psi}{}} } 
\]
where the $\Pi_i$ do not contain boxed formulas. 
We distinguish the cases that the width of the cut is zero and that it is bigger than zero. In the first case, there is a cutfree proof $\cald_3$ of $(\Pi_1,\dbx\Ga_1 \seq \phi)$ by Lemma~\ref{lemnoncritical}. Therefore in the following derivation all cuts are either of lower degree or of the same degree and width but of lower level than {\it Cut}$_1$.
\begin{equation*}\small 
   \infer[{\it Cut}_2]{\Pi_1,\Pi_1,\Pi_2,\bx\Ga_1,\dbx\Ga_1,\dbx\Ga_2,\bx\psi \seq \psi}{
    \deduce[\cald_3]{\Pi_1,\dbx\Ga_1\seq \phi}{} & 
    \infer[{\it Cut}_3]{\phi,\Pi_1,\Pi_2,\bx\Ga_1,\dbx\Ga_2,\bx\psi \seq \psi}{
     \infer{\Pi_1,\bx\Ga_1 \seq \bx\phi}{
      \deduce[\cald_1']{\Pi_1,\dbx\Ga_1,\bx\phi \seq \phi}{} } & 
     \deduce[\cald_2']{\dbx\phi,\Pi_2,\dbx\Ga_2,\bx\psi \seq \psi}{} } }   
\end{equation*}
Thus the induction hypothesis can be applied, which in combination with the closure under contraction yields a cutfree proof of $(\Pi_1,\Pi_2,\dbx\Ga_1,\dbx\Ga_2,\bx\psi \seq \psi)$. An application of $\rsch_{\rm SL}$ gives the desired 
cutfree proof of $(\bx\Sig_1,\bx\Sig_2,\Pi_1,\Pi_2,\bx\Ga_1,\bx\Ga_2 \seq \bx\psi)$.

Finally, we treat the case in which the width of {\it Cut}$_1$ is bigger than zero. Figure~\ref{figcut} might help to understand all steps that we will take. The width of {\it Cut}$_1$ is bigger than zero, so the derivation of the left premise contains the following 
$\bx\phi$-critical branch segment, where part $\calb$ does not contain applications of $\rsch_{\rm SL}$ and $\Ga_1' \subseteq \Ga_1$ by Remark~\ref{remcritical}:
\[\small 
  \infer[{\it Cut}_1]{\bx\Sig_1,\bx\Sig_2,\Pi_1,\Pi_2,\bx\Ga_1,\bx\Ga_2 \seq \bx\psi}{
   \infer[\rsch_{\rm SL}]{\bx\Sig_1,\Pi_1,\bx\Ga_1\seq \bx\phi}{
    \deduce[\calb]{\Pi_1,\dbx\Ga_1,\bx\phi \seq \phi}{
     \infer[\rsch_{\rm SL}]{\bx\Sig_3,\Pi_3,\bx\Ga_1,\bx\Ga_3,\bx\phi \seq \bx\varchi}{
      \deduce{\Pi_3,\dbx\Ga_1',\dbx\Ga_3,\dbx\phi,\bx\varchi \seq \varchi}{\vdots} } } } 
  & 
  \deduce{\bx\Sig_2,\Pi_2,\bx\Ga_2,\bx\phi\seq \bx\psi}{\cald_2} }
\]
Recall that $\cald_1$ is the derivation in $\cald$ of $(\bx\Sig_1,\Pi_1,\bx\Ga_1\seq \bx\phi)$ and $\overline{\cald}_1$ is by definition equal to $\cald_1$ except for the last inference, which is still an application of $\rsch_{\rm SL}$, 
but one in which~$\bx\Sig_1$ is not introduced:
\[\small 
   \infer[\rsch_{\rm SL}]{\Pi_1,\bx\Ga_1\seq \bx\phi}{
    \deduce[\calb]{\Pi_1,\dbx\Ga_1,\bx\phi \seq \phi}{
     \infer[\rsch_{\rm SL}]{\bx\Sig_3,\Pi_3,\bx\Ga_1,\bx\Ga_3,\bx\phi \seq \bx\varchi}{
      \deduce{\Pi_3,\dbx\Ga_1',\dbx\Ga_3,\dbx\phi,\bx\varchi \seq \varchi}{\vdots} } } } 
\]
Let $S_3$ be the sequent at the top of the segment, $S_3=(\Pi_3,\dbx\Ga_1',\dbx\Ga_3,\dbx\phi,\bx\varchi \seq \varchi)$.
Since there is no $\rsch_{\rm SL}$-inference in $\calb$, this means that $g_{\overline{\cald}_1}(S_3)=2$, all sequents in $\calb$ have $\rsch_{\rm SL}$-grade 1, and the endsequent of $\overline{\cald}_1$ has $\rsch_{\rm SL}$-grade 0.  
Therefore by Lemma~\ref{lemweakening}, $(\overline{\cald}_1)_{\footnotesize\varchi}$ is a derivation in which the 
$\bx\phi$-critical branch segment becomes the following, where $\calb'$ does not contain applications of $\rsch_{\rm SL}$ and $k$ equals 1 or 2 depending on whether $\varchi$ is boxed or not:
\[\small 
  \infer[\rsch_{\rm SL}]{\Pi_1,\bx\Ga_1,\bx\varchi\seq \bx\phi}{
   \deduce[\calb']{\Pi_1,\dbx\Ga_1,\bx\varchi, \varchi,\bx\phi \seq \phi}{
    \infer[\rsch_{\rm SL}]{\bx\Sig_3,\Pi_3,\bx\Ga_1,\bx\Ga_3,\bx\varchi,\varchi,\bx\phi \seq \bx\varchi}{
    \Pi_3,\dbx \Ga_1',\dbx\Ga_3,\dbx\phi,\bx\varchi,\varchi^k,\bx\varchi \seq \varchi} } } 
\]
Since $(S_3)_{\footnotesize\varchi}$ has $\varchi$ in its antecedent, the sequent $S_4=(\Pi_3,\dbx\Ga_1',\dbx\Ga_3,\bx\varchi,\varchi^k,\bx\varchi \seq \varchi)$ that one obtains by removing $\dbx\phi$ from the antecedent of $(S_3)_{\footnotesize \varchi}$ is still derivable and clearly has a cutfree derivation, say $\cald_3'$. 
Let $\cald_3$ be the derivation in $(\overline{\cald}_1)_{\footnotesize\varchi}$ of $S_3$ and let $\cald_1^\circ$ be the result of replacing in $(\overline{\cald}_1)_{\footnotesize\varchi}$ derivation $\cald_3$ by $\cald_3'$:
\[\small 
 (\overline{\cald}_1)_{\footnotesize\varchi}
  \left\{
   \begin{array}{c}
    \infer[\rsch_{\rm SL}]{\Pi_1,\bx\Ga_1,\bx\varchi\seq \bx\phi}{
     \deduce[\calb']{\Pi_1,\dbx\Ga_1,\bx\varchi, \varchi,\bx\phi \seq \phi}{
      \infer[\rsch_{\rm SL}]{\bx\Sig_3,\Pi_3,\bx\Ga_1,\bx\Ga_3,\bx\varchi,\varchi,\bx\phi \seq \bx\varchi}{
       \deduce{\Pi_3,\dbx\Ga_1',\dbx\Ga_3,\dbx\phi,\bx\varchi,\varchi^k,\bx\varchi \seq \varchi}{\cald_3}} } } 
   \end{array}
  \right.
  \ \ \ \ 
 \cald_1^\circ 
  \left\{
   \begin{array}{c}
    \infer[\rsch_{\rm SL}]{\Pi_1,\bx\Ga_1,\bx\varchi\seq \bx\phi}{
     \deduce[\calb']{\Pi_1,\dbx\Ga_1,\bx\varchi, \varchi,\bx\phi \seq \phi}{
      \infer[\rsch_{\rm SL}]{\bx\Sig_3,\Pi_3,\bx\Ga_1,\bx\Ga_3,\bx\varchi,\varchi,\bx\phi \seq \bx\varchi}{
       \deduce{\Pi_3,\dbx\Ga_1',\dbx\Ga_3,\bx\varchi,\varchi^k,\bx\varchi \seq \varchi}{\cald_3'}} } } 
   \end{array}
  \right.
\]
Observe that $\cald_1^\circ$ is still a cutfree derivation, that its endsequent is still 
$(\Pi_1,\bx\Ga_1,\bx\varchi\seq \bx\phi)$, and that it contains one $\bx\phi$-critical inference less than $\cald_1$.

Consider the following proof: 
\[\small 
  \infer[{\it Cut}_2]{\Pi_1^3,\Pi_3,\bx\Ga_1^2,\dbx\Ga_1,\dbx\Ga_1',\dbx\Ga_3,\bx\varchi^3 \seq \varchi}{
   \infer[{\it Cut}_3]{\Pi_1^2,\bx\Ga_1,\dbx\Ga_1,\bx\varchi \seq \phi}{
    \deduce[\cald_1^\circ]{\Pi_1,\bx\Ga_1,\bx\varchi\seq \bx\phi}{} 
    & 
    \deduce[\cald_1']{\Pi_1,\dbx\Ga_1,\bx\phi\seq \phi}{}} 
   &
   \infer[{\it Cut}_4]{\phi,\Pi_1,\Pi_3,\bx\Ga_1,\dbx\Ga_1',\dbx\Ga_3,\bx\varchi^2\seq \varchi}{
      \deduce[\cald_1^\circ]{\Pi_1,\bx\Ga_1,\bx\varchi\seq \bx\phi}{}  
    &
    \deduce[\cald_3]{\Pi_3,\dbx\Ga_1',\dbx\Ga_3,\dbx\phi,\bx\varchi\seq\varchi}{} }  
   }  
\]
Given that $\cald_1',\cald_1^\circ,\cald_3$ are cutfree, the only cuts in the proof are {\it Cut}$_2$, {\it Cut}$_3$, and {\it Cut}$_4$. By the observation above, the widths of {\it Cut}$_3$ and {\it Cut}$_4$ are lower than that of {\it Cut}$_1$. Thus there exist cutfree proofs of the two premises of {\it Cut}$_2$. Since {\it Cut}$_2$ has lower degree than {\it Cut}$_1$ the induction hypothesis applies. Together with the closure under contraction, and the fact that $\Ga_1' \subseteq \Ga_1$, this proves that there exists a cutfree proof, say $\cald_4$, of $S_5 = (\Pi_1,\Pi_3,\dbx\Ga_1,\dbx\Ga_3,\bx\varchi \seq \varchi)$. 

Observe that in 
$(\cald_1)^{\footnotesize\Pi_1}$, the segment becomes the following, where $\calb''$ does not contain applications of $\rsch_{\rm SL}$:
\[\small 
  \infer[\rsch_{\rm SL}]{\bx \Sigma_1,\Pi_1^2,\bx\Ga_1 \seq \bx\phi}{
   \deduce[\calb'']{\Pi_1^2,\dbx\Ga_1,\bx\phi \seq \phi}{
    \infer[\rsch_{\rm SL}]{
     \bx\Sig_3,\Pi_1,\Pi_3,\bx \Ga_1,\bx\Ga_3,\bx\phi \seq \bx\varchi}{
     \Pi_1,\Pi_3,\dbx\Ga_1',\dbx\Ga_3,\dbx\phi,\bx\varchi \seq \varchi} } } 
\]
Let $\cald_1^\triangledown$ denote the result of replacing, in $(\cald_1)^{\footnotesize\Pi_1}$, the derivation of the sequent at the top of the segment by derivation $\cald_4$ of $S_5$. And let $\cald'$ be the result of replacing $\cald_1$ in $\cald$ by $\cald_1^\triangledown$. Thus in $\cald'$ the segment, the part above the segments and the last inference become 
\[\small 
  \infer[{\it Cut}_5]{\bx\Sig_1,\bx\Sig_2,\Pi_1^2,\Pi_2,\bx\Ga_1,\bx\Ga_2 \seq \bx\psi}{
   \infer[\rsch_{\rm SL}]{\bx\Sigma_1,\Pi_1^2,\bx\Ga_1 \seq \bx\phi}{
    \deduce[\calb'']{\Pi_1^2,\dbx\Ga_1,\bx\phi \seq \phi}{
     \infer[\rsch_{\rm SL}]{
      \bx\Sig_3,\Pi_1,\Pi_3,\bx\Ga_1,\bx\Ga_3,\bx\phi \seq \bx\varchi}{
       \deduce{\Pi_1,\Pi_3,\dbx\Ga_1,\dbx\Ga_3,\bx\varchi \seq \varchi}
        {\cald_4} } } } 
  &
  \deduce[\cald_2]{\bx\Sig_2,\Pi_2,\bx\Ga_2,\bx\phi \seq \bx\psi}{} }
\]
Note that $\cald'$ still consists of valid inferences. 
The $\bx\phi$ at the top of the segment is now introduced in the first inference of the segment, and therefore no longer principal in that inference. Thus the width of {\it Cut}$_5$ is lower than that of {\it Cut}$_1$.

\begin{landscape}
\begin{figure}
\[\small
 \infer=[{\it Contraction}]{\bx\Sig_1,\bx\Sig_2,\Pi_1,\Pi_2,\bx\Ga_1,\bx\Ga_2 \seq \bx\psi}{
 \infer[{\it Cut}_5]{\bx\Sig_1,\bx\Sig_2,\Pi_1^2,\Pi_2,\bx\Ga_1,\bx\Ga_2 \seq \bx\psi}{
   \infer[\rsch_{\rm SL}]{\bx\Sig_1,\Pi_1^2,\bx\Ga_1 \seq \bx\phi}{
    \deduce[\calb'']{\Pi_1^2,\dbx\Ga_1,\bx\phi \seq \phi}{
     \infer[\rsch_{\rm SL}]{
      \bx\Sig_3,\Pi_1,\Pi_3,\bx\Ga_1,\bx\Ga_3,\bx\phi \seq \bx\varchi}{
       \infer=[{\it Weakening/Contraction}]{\Pi_1,\Pi_3,\dbx\Ga_1,\dbx\Ga_3,\bx\varchi \seq \varchi}{ 
        \infer[{\it Cut}_2]{\Pi_1^3,\Pi_3,\bx\Ga_1^2,\dbx\Ga_1,\dbx\Ga_1',\dbx\Ga_3,\bx\varchi^3 \seq \varchi}{
   \infer[{\it Cut}_3]{\Pi_1^2,\bx\Ga_1,\dbx\Ga_1,\bx\varchi \seq \phi}{
     \infer[\rsch_{\rm SL}]{\Pi_1,\bx\Ga_1,\bx\varchi\seq \bx\phi}{
     \deduce[\calb']{\Pi_1,\dbx\Ga_1,\bx\varchi, \varchi,\bx\phi \seq \phi}{
      \infer[\rsch_{\rm SL}]{\bx\Sig_3,\Pi_3,\bx\Ga_1,\bx\Ga_3,\bx\varchi,\varchi,\bx\phi \seq \bx\varchi}{
       \deduce{\Pi_3,\dbx\Ga_1',\dbx\Ga_3,\bx\varchi,\varchi^k,\bx\varchi \seq \varchi}{\cald_3'}} } } 
    & 
    \deduce[\cald_1']{\Pi_1,\dbx\Ga_1,\bx\phi\seq \phi}{}} 
   &
   \infer[{\it Cut}_4]{\phi,\Pi_1,\Pi_3,\bx\Ga_1,\dbx\Ga_1',\dbx\Ga_3,\bx\varchi^2\seq \varchi}{
       \infer[\rsch_{\rm SL}]{\Pi_1,\bx\Ga_1,\bx\varchi\seq \bx\phi}{
     \deduce[\calb']{\Pi_1,\dbx\Ga_1,\bx\varchi, \varchi,\bx\phi \seq \phi}{
      \infer[\rsch_{\rm SL}]{\bx\Sig_3,\Pi_3,\bx\Ga_1,\bx\Ga_3,\bx\varchi,\varchi,\bx\phi \seq \bx\varchi}{
       \deduce{\Pi_3,\dbx\Ga_1',\dbx\Ga_3,\bx\varchi,\varchi^k,\bx\varchi \seq \varchi}{\cald_3'}} } }    
    &
    \deduce[\cald_3]{\Pi_3,\dbx\Ga_1',\dbx\Ga_3,\dbx\phi,\bx\varchi\seq\varchi}{} }  
   }  
 } 
    } } } 
   &
  \deduce{\hspace{-3cm}\bx\Sig_2,\Pi_2,\bx\Ga_2,\bx\phi \seq \bx\psi}{\hspace{-1.5cm}\cald_2} }
  }
\]
\caption{Illustration of the main steps in the proof of Theorem~\ref{thmcutadm}. It does not contain $\cald_4$, but the proof that $\cald_4$ replaces, and it is not a proof, because it contains weakening and contraction steps that are not part of the calculus.}
 \label{figcut}
\end{figure}
\end{landscape}

Since their degrees are the same, the induction hypothesis can be applied to obtain a cutfree proof of the endsequent. 
Contraction implies that $(\bx\Sig_1,\bx\Sig_2,\Pi_1,\Pi_2,\bx\Ga_1,\bx\Ga_2 \seq \bx\psi)$ has a cutfree proof as well, which is what we had to show. 
\end{proof}

\section{Termination} 
 \label{sectermination}

Proof search in calculus \DY\ is terminating, which is one of the main reasons for its development. There are various concepts of termination. In this paper we use a very strict version, close to strong termination. A sequent calculus is \emph{strongly terminating} if and only if there is a well-founded ordering on sequents such that for all rules in the calculus the premises are smaller in this ordering than the conclusion. Decidability of the logic immediately follows from this. For $\DY$, we use the ordering $\ll$ as defined in Section \ref{seclogics}. It makes it possible to establish the equivalence between $\LJ$ and $\DY$ \citep{dyckhoff92}. Note that proof search in $\LJ$ is not strongly terminating but \emph{weakly terminating}, i.e. there exists a terminating process of deciding the derivability of a sequent involving a global check on sequents in the proof search \citep{troelstra98}.

Although the calculi $\DYK$ and $\DYKD$ from \citep{iemhoff18} are strongly terminating, this is no longer the case for $\DYGL$ and $\DYSL$. However, based on a termination proof for a sequent calculus for $\GL$ in \citep{bilkova06}, it is shown in \citep{giessen&iemhoff2019} that $\DYGL$ has a property very close to strong termination: if, in the process of bottom-up proof search, we stop searching along a branch whenever we reach a sequent $S$ for which $S^a\cap S^c$ is not empty, then the standard proof search always terminates. As in Lemma~\ref{lemstructural}, let us call such a sequent an {\em extended axiom}. We call this \emph{termination modulo extended axioms}. 
Using this, we can establish the equivalence between $\LJSL$ and $\DYSL$ as well.

\begin{theorem}\label{thmtermination}
Proof search in $\DYSL$ terminates modulo extended axioms.
\end{theorem}
\begin{proof}
Let $S$ be the sequent for which we perform the proof search. Let $c$ be the number of all boxes occurring in $S$. For any rule application in the proof search each premise is either an extended axiom or comes before the conclusion in the order $\ord^c$ defined in Section~\ref{seclogics}. The proof is completely analogous to the termination proof of $\DYGL$ in \citep{giessen&iemhoff2019}. 
Let us only look into the details for rule $\leftimprule_{\rm SL}$. Its conclusion is of the form $(\Pi, \Box \Gamma, \Box \phi \impl \psi \Impl \Delta)$ and its premises are $(\Pi, \boxdot \Gamma, \Box \phi, \Box \phi \impl \psi \Impl \phi)$
 and $(\Pi, \Box \Gamma, \psi \Impl \Delta)$ where $\Pi$ does not contain boxed formulas. It is straightforward that the second premise is smaller than the conclusion of the rule with respect to $\ord^c$. For the first premise we distinguish two cases. If $\Box \phi \in \Box \Gamma$, then the first premise is an extended axiom. If $\Box \phi \notin \Box \Gamma$, the antecedent of the first premise has more boxed formulas than the antecedent of the conclusion and, thus, this premise is smaller than the conclusion.
\end{proof}

In rule $\rsch_{\rm SL}$ we assume that $\Pi$ does not contain boxed formulas, but we might drop this condition while keeping termination modulo extended axioms as discussed in Lemma~\ref{lem systems equivalent}. Important to note is that the condition on the $\Pi$'s in rule $\impl_{\rm SL}$ is necessary for proving termination.

\section{Equivalence of $\LJSL$ and $\DYSL$}
 \label{secequivalence}
 
In this section we prove that $\LJSL$ and $\DYSL$ are equivalent, meaning that the calculi derive the same sequents. This implies that cut-elimination is admissible in the latter, a fact that can be useful in the application of these calculi. First, we prove, in the next section, a normal form theorem for proofs in the calculus $\LJSL$, a fact that enables us to establish, in the section after that, the equivalence between the two calculi.  

\subsection{Strict proofs in $\LJSL$}
A multiset is {\em irreducible} if it has no element that is a disjunction or a conjunction or falsum and for no atom $p$ does it contain both $p\imp\psi$ and $p$. A sequent $S$ is {\em irreducible} if $S^a$ is. A proof is {\em sensible} if its  last inference does not have a principal formula on the left of the form $p\imp\psi$ for some atom $p$ and formula $\psi$.\footnote{In \cite{iemhoff18} the requirement that the principal formula be on the left was erroneously omitted.} A proof in $\LJSL$ is {\em strict} if in the last inference, in case it is an instance of L$\imp$ with principal formula $\bx\phi \imp \psi$, the left premise is an axiom or the conclusion of an application of the modal rule $\rsch_{\rm SL}$. Note that because the formula in the succedent of the left premise is $\bx\phi$, in case the left premise is an instance of an axiom, it can only be an instance of $L\bot$. This implies that if $S$ is irreducible, the left premise cannot be an instance of an axiom and thus is required to be the conclusion of an application of a modal rule.

\begin{lemma} 
 \label{lemstrict} 
In $\LJSL$, every irreducible sequent that is provable has a sensible strict proof. 
\end{lemma}
\begin{proof}
This is proved in the same way as the corresponding lemma (Lemma 1) in \citep{dyckhoff92}. 
Arguing by contradiction, assume that among all provable irreducible sequents that have no sensible strict proofs, $S$ is such a sequent with the shortest proof, $\cald$, where the {\em length} of a proof is the length of its leftmost branch. Thus the last inference in the proof is an application 
\[
 \infer{\Ga,\phi\imp \psi \seq \De}{
 \deduce[\cald_1]{\Ga, \phi\imp \psi \seq \phi}{} & 
 \deduce[\cald_2]{\Ga, \psi \seq \De}{} }
\]
of L$\imp$, where $\phi$ is an atom or a boxed formula.  
Since $S^a$ is irreducible, $\bot \not\in S^a$ and if $\phi$ is an atom, $\phi\not\in S^a$. Therefore the left premise cannot be an axiom and hence is the conclusion of a rule, say $\rsch$. Since the succedent of the conclusion of $\rsch$ consists of an atom or a boxed formula, $\rsch$ is a left rule or $\rsch_{\rm SL}$. The latter case cannot occur, since the proof then would be strict and sensible. Thus $\rsch$ is a left rule. 

We proceed as in \citep{dyckhoff92}. Sequent $(\Ga, \phi\imp \psi \seq \phi)$ is irreducible and has a shorter proof than $S$. Thus its subproof $\cald_1$ is strict and sensible. Since the sequent is irreducible and $\phi$ is an atom or a boxed formula, the last inference of $\cald_1$ is L$\imp$ with a principal formula $\phi'\imp \psi'$ such that $\phi'$ is not an atom. Let $\cald'$ be the proof of the left premise $(\Ga, \phi\imp \psi \seq \phi')$. Thus the last part of $\cald$ looks as follows, where $\Pi,\phi'\imp \psi'=\Ga$.  
\[
 \infer{\Pi, \phi\imp \psi,\phi'\imp \psi' \seq \De}{
 \infer{\Pi, \phi\imp \psi,\phi'\imp \psi' \seq \phi}{
  \deduce[\cald']{\Pi, \phi\imp \psi,\phi'\imp \psi' \seq \phi'}{} & 
  \deduce[\cald'']{\Pi, \phi\imp \psi,\psi' \seq \phi}{}} & 
 \deduce[\cald_2]{\Pi, \psi,\phi'\imp \psi' \seq \De}{} }
\]
Consider the following proof of $S$.  
\[
 \infer{\Pi, \phi\imp \psi,\phi'\imp \psi' \seq \De}{
  \deduce[\cald']{\Pi, \phi\imp \psi,\phi'\imp \psi' \seq \phi'}{} & 
  \infer{\Pi, \phi\imp \psi,\psi'\seq \De}{ 
   \deduce[\cald'']{\Pi, \phi\imp \psi,\psi' \seq \phi}{} & 
   \deduce[\cald''']{\Pi,\psi,\psi' \seq \De}{} 
  } 
 }
\]
The existence of $\cald'''$ follows from Lemma~\ref{lemstructural} (Inversion) and the existence of $\cald_2$. The obtained proof is strict and sensible: In case $\phi'$ is not a boxed formula, this is straightforward. In case $\phi'$ is a boxed formula, it follows from the fact, observed above, that $\cald_1$ is strict and sensible. 
\end{proof}

\subsection{Equivalence}

\begin{theorem} [Equivalence]
 \label{thmequivalence}
\\
$\af_{\LJSL} S$ if and only if $\af_{\DYSL} S$. 
\end{theorem}
\begin{proof}
The proof is an adaptation of the proof of Theorem 1 in \citep{dyckhoff92}. 

The direction from right to left is straightforward because $\LJSL$ is closed under the structural rules and {\it Cut}, but let's fill in some of the details. We use induction to the height of the proof of a sequent in $\DYSL$. Suppose $\af_{\DYSL}S$. If $S$ is an instance of an axiom, then clearly $\af_{\LJSL}S$ as well. Suppose $S$ is not an instance of an axiom and consider the last inference of the proof of $S$. We distinguish according to the rule $\rsch$ of which the last inference is an instance. 

If the rule is $Lp\!\imp$, then $S$ is of the form $\Ga,p,p\imp \phi \seq \De$. The premise is 
$\Ga,p,\phi \seq \De$, which, by the induction hypothesis, is derivable in $\LJSL$. It is not hard to show that $\Ga,p,p\imp \phi\seq \phi$ is also derivable in $\LJSL$. Applications of Cut and Contraction show that so is $S$.   

If the rule is $L\!\imp\!\imp$, then $S$ is of the form $\Ga, (\phi \imp \psi)\imp \gam \seq \De$
and the premises are $\Ga,\gam\seq \De$ and $\Ga,\psi\imp\gam\seq \phi\imp\psi$. The premises are derivable in $\LJSL$ by the induction hypothesis. It is not difficult to show that then $\Ga,(\phi \imp \psi) \imp \gam, \phi \seq \psi$ is derivable in~$\LJSL$ as well. Hence so is $\Ga,(\phi \imp \psi) \imp \gam\seq \phi \imp \psi$. An application of $L\imp$ proves that $S$ is derivable in $\LJSL$. 

If the rule is $\leftimprule_{\rm SL}$, then $S$ is of the form $(\Pi, \Box \Gamma, \Box \phi \impl \psi \Impl \Delta)$
and the premises are $(\Pi, \boxdot \Gamma, \Box \phi, \Box \phi \impl \psi \Impl \phi)$ and $(\Pi, \Box \Gamma, \psi \Impl \Delta)$. The premises are derivable in $\LJSL$ by the induction hypothesis. The following derivation shows that $S$ is derivable in $\LJSL$:
\begin{center}
\AXC{$\Pi, \boxdot \Gamma, \Box \phi, \Box \phi \impl \psi \Impl \phi$}
\RL{$\rsch_{\rm SL}$}
\UIC{$\Pi, \Box \Gamma, \Box \phi \impl \psi \Impl \Box \phi$}
\AXC{$\Pi, \Box \Gamma, \psi \Impl \Delta$}
\RL{$L\!\imp$}
\BIC{$(S)  \ \Pi, \Box \Gamma, \Box \phi \impl \psi \Impl \Delta$}
\DisplayProof
\end{center}

The remaining cases are left to the reader.
 
For the other direction we have to show that every sequent $S$ that is provable in $\LJSL$ is provable in $\DYSL$. This is proved by induction on the well-ordering $\ord^{c}$, 
where $c$ is the number of boxed formulas that occur in formulas in $S$ as defined in Section \ref{seclogics}. From Theorem~\ref{thmtermination} we know that the proof search on $S$ in $\DYSL$ terminates modulo extended axioms in this ordering. Let $\ord$ denote $\ord^{c}$ in the rest of the proof and suppose $\af_{\LJSL}S$. 

Sequents lowest in the ordering $\ord$ do not contain connectives and no boxed formula in the succedent. Thus in this case $S$ has to be an instance of an axiom, and since~$\LJSL$ and~$\DYSL$ have the same axioms, $S$ is provable in $\DYSL$. 

We turn to the case that $S$ is not the lowest in the ordering $\ord$. 
If $S^a$ contains a conjunction,  $S = (\Ga,\phi_1\en\phi_2 \seq \De)$, then $S'=(\Ga,\phi_1,\phi_2\seq \De)$ is provable in $\LJSL$ too by Inversion (Lemma~\ref{lemstructural}). As $S'\ord S$, $S'$ is provable in $\DYSL$ by the induction hypothesis. Thus so is $(\Ga,\phi_1\en\phi_2 \seq \De)$. A disjunction in $S^a$ as well as the case that both $p$ and $p\imp \phi$ belong to~$S^a$, can be treated in the same way. 

Thus only the case that $S$ is irreducible remains, and by Lemma~\ref{lemstrict} we may assume its proof to be sensible and strict. Thus its last inference is an application of a rule, $\rsch$, that is either a nonmodal right rule, $\rsch_{\rm SL}$ or $L\!\imp$. In the first case, $\rsch$ belongs to both calculi and the premise of $\rsch$ is lower in the $\sqsubset$ ordering than $S$ and thus the induction hypothesis applies. 

In the case of $\rsch_{\rm SL}$, let $S=(\bx\Sig,\Pi,\bx\Ga \seq \bx\phi)$ and let 
$(\Pi,\dbx\Ga,\bx\phi \seq \phi)$ be the premise of the inference where $\Pi$ does not contain boxed formulas. There are two cases: either $\bx\phi$ does occur in $\bx\Sig$ or $\Box \Gamma$, or it does not. In the first case, $S$ is an extended axiom which is also derivable in $\DYSL$. In the latter case, we consider sequent $S'=(\boxdot\Sig,\Pi,\dbx\Ga,\bx\phi \seq \phi)$ which is derivable in $\LJSL$ by weakening (Lemma~\ref{lemstructural}). We have $S' \sqsubset S$, because $S'$ has more boxed formulas in its antecedent counted as a set than the antecedent of $S$. So $S'$ is by the induction hypothesis derivable in $\DYSL$. An application of $\rsch_{\rm SL}^4$ shows that $S$ is derivable in $\DYSL$.

We turn to rule $L\!\imp$. 
Suppose that the principal formula of the last inference is $(\gam \imp \psi)$ and $S = (\Ga,\gam\imp\psi \seq \De)$. Since the proof is sensible, $\gam$ is not atomic. 
We distinguish according to the main connective of $\gam$.  

If $\gam = \bot$, then $(\Ga\seq \De)$ is derivable in $\LJSL$ (follows easily from admissibility of {\it Cut}), and therefore in $\DYSL$. As $\DYSL$ is closed under weakening, Lemma~\ref{lemstructural}, $S$ is derivable in~$\DYSL$ too. 

If $\gam = \phi_1 \en \phi_2$, then $(\Ga,\phi_1 \imp (\phi_2\imp \psi) \seq \De)$ is derivable in $\LJSL$ (using {\it Cut}). 
Thus it is derivable in $\DYSL$ by the induction hypothesis. Hence so is 
$(\Ga,\phi_1 \en \phi_2 \imp \psi\seq \De)$. The case that $\gam = \phi_1 \of \phi_2$ is analogous.

If $\gam = \phi_1 \imp \phi_2$, then because $\gam\imp\psi$ is the principal formula, 
both sequents $(\Ga, \gam \imp \psi \seq \gam)$ and $(\Ga, \psi \seq \De)$ are derivable in $\LJSL$. Thus so is sequent $(\Ga,\phi_2\imp\psi \seq \phi_1\imp\phi_2)$ (using {\it Cut} and Inversion of the $R\!\impl$ rule (Lemma~\ref{lemstructural})). Since this sequent and 
$(\Ga, \psi \seq \De)$ are lower in the ordering $\sqsubset$ than $S$, they are 
derivable in $\DYSL$ by the induction hypothesis. Hence so is $S$. 

If $\gam = \bx\phi$, the fact that the proof is strict implies that the left premise is the conclusion of an application of $\rsch_{\rm SL}$, so the proof in $\LJSL$ looks as follows, where $\Gamma = \Box \Sigma \cup \Pi \cup \Box \Gamma'$ for some $\Ga',\Sig,\Pi$ such that $\Pi$ does not contain boxed formulas:
\begin{center}
\AXC{$\Pi, \boxdot \Gamma', \Box \phi \impl \psi, \Box \phi \Impl \phi$}
\RL{$\rsch_{\rm SL}$}
\UIC{$\Box \Sigma, \Pi, \Box \Gamma', \Box \phi \impl \psi \Impl \Box \phi$}
\AXC{$\Box \Sigma, \Pi, \Box \Gamma', \psi \Impl \Delta$}
\RL{$L\!\imp$}
\BIC{$(S) \ \Box \Sigma, \Pi, \Box \Gamma', \Box \phi \impl \psi \Impl \Delta$}
\DisplayProof
\end{center}
The right premise is smaller in $\ord$ than $S$, so by induction hypothesis  $(\Box \Sigma, \Pi, \Box \Gamma', \psi \Impl \Delta)$ is also provable in $\DYSL$. There are two cases: either $\Box \phi$ does occur in $\Box \Sigma$ or in $\Box \Gamma$, or not. We show that in both cases the sequent $S' = (\boxdot \Sigma, \Pi, \boxdot \Gamma', \Box \phi \impl \psi, \Box \phi \Impl \phi)$ is provable in $\DYSL$, and, in turn, an application of rule $\imp_{\rm SL}$ shows that $S$ is also provable in $\DYSL$. In the first case, $S'$ is an extended axiom and so it is derivable in $\DYSL$. In the latter case, note that $S' \ord S$ and $S'$ is provable in $\LJSL$ by weakening (Lemma \ref{lemstructural}), so by induction hypothesis we have that $S'$ is provable in $\DYSL$.
\end{proof}

From the previous theorem and Lemma~\ref{lemstructural} the following can be obtained. 

\begin{corollary} 
 \label{corconcutdy} 
The rules {\it Cut} and {\it Contraction} are admissible in $\DYSL$.
\end{corollary}

Since $\DYSL$ admits cut, it is easy to see that it is sound and complete with respect to logic $\iSL$. Therefore, since $\DYSL$ is terminating modulo extended axioms, we obtain the following corollary.

\begin{corollary}
Logic $\iSL$ is decidable.
\end{corollary}

\section{Interpolation}
 \label{secinterpolation}
A logic $\lgc$ has {\em (Craig) interpolation} if whenever $\phi\imp\psi$ is derivable in the logic there is a formula $\varchi$ such that all atoms that occur in $\varchi$ occur both in $\phi$ and in $\psi$, and 
\[
 \af_\lgc \phi\imp\varchi \ \ \ \ \af_\lgc \varchi\imp\psi.
\]

\begin{theorem}
For every sequent $(\Ga_1,\Ga_2\seq \De)$ derivable in $\LJSL$ there exists a formula $I$ such that all atoms that occur in $I$ occur both in $\Ga_1$ and in $\Ga_2\cup\De$ and such that 
\[
 \af_{\LJSL}\Ga_1\seq I \ \ \ \ \af_{\LJSL}\Ga_2,I\seq \De.
\]
\end{theorem}
\begin{proof}
Analogous to the proof of interpolation for $\LJGL$ in \citep{giessen&iemhoff2019}. 
\end{proof}

The same result holds for $\DYSL$. The following has been proven, but not yet published, independently by Litak and Visser, building on their work on the Lewis arrow in the setting of intuitionistic logic \citep{litak&visser18}: 

\begin{corollary}
$\iSL$ has Craig interpolation. 
\end{corollary}

\section{Kripke semantics for $\iSL$}\label{Kripke iSL}

In this section we examine the Kripke semantics for $\iSL$. Kripke semantics for $\iSL$ has been introduced in \citep{litak14} and \citep{ardeshir&mojtahedi18}. Here we focus on a countermodel construction in the proof of completeness for a variant of $\DYSL$, denoted~$\DYSLprime$, which is also terminating. Closely related are the works of \cite{svejdar06} and \cite{avron84}. \v{S}vejdar gives a countermodel construction for $\DY$ for $\IPC$. Avron provides a semantic proof of the cut-elimination theorem for the standard sequent calculus for classical~${\sf GL}$. Completeness with respect to the Kripke semantics provides a semantic proof of the admissibility of the cut rule.

The only difference between $\DYSL$ and $\DYSLprime$ is another representation of rule $L\!\imp\!\imp$ that enforces an immediate backward application of $R\!\imp$ to the left premise. So $\DYSLprime$ is defined as $\DYSL$ where $L\!\imp\!\imp$ is replaced by the following rule:
\begin{center}
\AXC{$\Gamma, \psi \impl \gamma, \phi \Impl \psi$}
\AXC{$\Gamma, \gamma \Impl \Delta$}
\RL{$L\!\imp\!\imp'$}
\BIC{$\Gamma, (\phi \impl \psi) \impl \gamma \Impl \Delta$}
\DisplayProof
\end{center}
This change is necessary in the proof of completeness. It is easy to see that the properties from Lemma~\ref{lemstructural} that hold for $\DYSL$ also hold for $\DYSLprime$. In addition, it is terminating modulo extended axioms similarly proved as Theorem~\ref{thmtermination}.

Kripke models for intuitionistic propositional logic are structures of the form $(W,\leq,V)$ where $W$ is the set of possible worlds structured by partial order $\leq$ and $V$ is a valuation determining which atoms are valid in which worlds. Kripke's possible worlds semantics are also used in modal logic. In the classical setting they have the form $(W,R,V)$ where~$R$ represents the modal relation. As \cite{simpson94} phrases it in his PhD thesis, it is `(mathematically) natural to build a semantics of intuitionistic modal logic by combining the two Kripkean accounts.' 

An \emph{intuitionistic modal Kripke model} is a structure $(W,\leq, R, V)$, where 
\begin{enumerate}[noitemsep]
\item $W$ is a non-empty set,
\item $\leq$ is a partial order on $W$ (\emph{intuitionistic relation}),
\item $R$ a binary relation (\emph{modal relation}),
\item $w \leq v R x$ implies $wRx$,
\item $V$ is a function $V: W \times Atoms \impl \{0,1\}$ satisfying \emph{monotonicity}, which means that $w \leq v$ and $V(w,p)=1$ implies $V(v,p)=1$. We call $V$ the \emph{valuation}.
\end{enumerate}
An $\iSL$\emph{-model} is an intuitionistic modal Kripke model satisfying the following criteria:
\begin{enumerate}[noitemsep]
\item[6.] relation $R$ is transitive, i.e., if $wRx$ and $xRy$, then $wRy$,
\item[7.] $R$ is conversely well-founded, i.e., there is no infinite $R$ path,
\item[8.] $wRx$ implies $w \leq x$.
\end{enumerate}

Points 1.-7. characterize $\iGL$-models. So point 8. is characteristic for $\iSL$-models. Note that for $\iSL$, point 6. follows from 4. and 8. We define a \emph{forcing relation} $\Vdash$ for a model $\calM = (W, \leq, R, V)$ as usual as follows:

\begin{listliketab} 
    \storestyleof{enumerate} 
        \begin{tabular}{Lll}
             & $\calM,w \Vdash p$ & iff $V(w,p)=1$,\\
             & $\calM,w \Vdash \bot$ & never,\\
             & $\calM,w \Vdash \phi \wedge \psi$ & iff $\calM,w \Vdash \phi$ and $\calM,w \Vdash \psi$,\\
             & $\calM,w \Vdash \phi \vee \psi$ & iff $\calM,w \Vdash \phi$ or $\calM,w \Vdash \psi$,\\
             & $\calM,w \Vdash \phi \impl \psi$ & iff for all $v \geq w$, $\calM,v \Vdash \phi$ implies $\calM, v \Vdash\psi$,\\
            & $\calM,w \Vdash \Box \phi$ & iff for all $x$ such that $wRx$ we have $\calM,x \Vdash \phi$.
        \end{tabular} 
\end{listliketab}

As usual, from this definition and point 5. it follows by induction on formula $A$ that $w \leq v$ and $\calM, w \Vdash A$ imply $\calM, v \Vdash A$.

Let $\calM = (W,\leq,R,V)$ be an intuitionistic modal Kripke model and $w \in W$. The \textit{submodel generated by} $w$ is the model $\calM' = (W', \leq', R', V')$ such that $w \in W'$ and if $x \in W'$ and $xRy$ or $x \leq y$, then $y \in W'$. Relations $\leq'$ and $R'$ are the restriction of $\leq$ and $R$ to $W'$ respectively, and valuation $V'$ is the restriction of $V$ to $\calM'$. We call $w$ the \textit{root} of the model. If $\calM$ is an $\iSL$-model, then so is $\calM'$. One can easily verify that if $x \in \calM'$, then $\calM, x \Vdash \phi$ if and only if $\calM',x \Vdash \phi$ for any formula $\phi$.

For the countermodel construction we define the
termination procedure modulo extended axioms in a more rigorous way imposing
an order on rule applications following \citep{bilkova06} where we first apply invertible rules and after that the non-invertible rules. Note that all rules in $\DYSLprime$ are invertible except for $L\!\imp\!\imp'$, $R\of$, $\impl_{{\rm SL}}$, and~$\rsch_{\rm SL}^{4}$. To make this explicit, we introduce the following concept.

We call a sequent \textit{saturated} in $\DYSLprime$ if it is not an extended axiom and it cannot be the conclusion of an invertible rule from $\DYSLprime$.

\begin{lemma}\label{lemirreducible}
A sequent of the form $\Pi, \Box \Gamma \Impl  \Delta$ is saturated if it satisfies the following conditions:
\begin{enumerate}[label=(\roman*)]
\item $\Delta$ is empty or its formula is of the form $p$, $\bot$, $ \phi_1 \vee  \phi_2$ or $\Box  \phi$,
\item all formulas in $\Pi$ have the form $p, p \impl  \phi , ( \phi  \impl  \psi ) \impl  \gamma $ or $\Box  \phi  \impl  \psi $,
\item not both $p \in \Pi$ and $p \impl  \phi  \in \Pi$,
\item not both $ \phi  \in \Pi \cup \Box \Gamma$ and $ \Delta = \{\phi \}  $.
\end{enumerate}
\end{lemma}
\begin{proof}
This is an easy observation by the form of the rules.
\end{proof}

Given a sequent $S =(\Gamma \Impl \Delta)$, the {\em proof search tree} of $S$ is defined as follows. We create a tree whose nodes are labeled by sequents. The root is labeled by $(\Gamma \Impl \Delta)$ and we apply the rules backwards. By backwards applying a rule to a sequent that is a label of a node in the tree, we create predecessor node(s) and label them by the premise(s) of the rule. We first apply all invertible rules, in arbitrary order. We continue, until no invertible rule can be applied (which is guaranteed to terminate by decreasing weight). If we end up with an extended axiom, then stop the search for that node. Otherwise, it is a saturated sequent and apply each possible non-invertible rule and create for each rule predecessor nodes of the node marked by the saturated sequent, and again label them by the premise(s) of the corresponding rule. Repeat the procedure for those nodes.

In addition we mark the sequents in the proof search tree as \emph{positive} or \emph{negative}. If a leave is labeled by an extended axiom or an instance of $L\bot$, mark it as \emph{positive}. If not, mark it as \emph{negative}. For other nodes, we move down the proof search tree marking nodes in the following way. A saturated sequent is marked as \emph{positive} if for at least one backwards applied rule all its corresponding predecessors have been marked as positive. Non-saturated sequents are marked as \emph{positive} if all its predecessors have been marked as positive.

Similarly to Theorem~\ref{thmtermination} we can prove the following.

\begin{theorem}
Sequent calculus $\DYSLprime$ terminates modulo extended axioms. In addition, according to the proof search described above, a sequent is provable in $\DYSLprime$ if it is marked as positive in the proof search tree.
\end{theorem}

\subsection{Soundness}
We write $\calM,w \Vdash \Gamma$ if $\calM,w \Vdash \phi$ for every $\phi \in \Gamma$. We say that sequent $\Gamma \Impl \Delta$ is \emph{valid in model $\calM$}, denoted $\calM \models \Gamma \Impl \Delta$, if for all $w$, $\calM, w \Vdash \Gamma$ implies $\calM, w \Vdash \Delta$. We say that $\Gamma \Impl \Delta$ is \emph{refuted in $\calM$} if it is not valid in $\calM$. And we say that sequent $\Gamma \Impl \Delta$ is \emph{valid}, denoted $\Gamma \models_{\iSL} \Delta$, if it is valid in all $\iSL$-models. And we say that sequent $\Gamma \Impl \Delta$ is \emph{refuted} if it is not valid. From now we write $\models$ to mean $\models_{\iSL}$.

\begin{theorem}[Soundness]
If $ \proofs_{\DYSLprime} \Gamma \Impl \Delta$, then $\Gamma \models \Delta$.
\end{theorem}
\begin{proof}
As usual, we use induction on the height $d$ of the derivation of $\Gamma \Impl \Delta$. We consider a few steps.

First consider $d=1$ and $\Gamma \Impl \Delta$ has the form $p,\Gamma' \Impl p$ for some $p$. Suppose $\calM, w \Vdash \phi$ for each $\phi \in \Gamma$. Then clearly $\calM,w\Vdash p$.

Consider $d>1$ and $\Gamma \Impl \Delta$ has the form $\bot, \Gamma' \Impl \Delta$. For each model $\calM$ we have $\calM,w \not \Vdash \bot$, hence $\bot, \Gamma' \Vdash \Delta$.

Now consider $d>1$. We treat four cases. First suppose the last rule applied is $R\wedge$, where $\Gamma \Impl \Delta$ has the form $\Gamma \Impl \phi \wedge \psi$ and 
\begin{center}
\AXC{$\Gamma \Impl \phi$}
\AXC{$\Gamma \Impl \psi$}
\RL{$R\wedge$}
\BIC{$\Gamma \Impl \phi \wedge \psi$}
\DisplayProof
\end{center}
Let $\calM$ be an $\iSL$-model and $w \in \calM$ such that $\calM,w \Vdash \Gamma$. By the induction hypothesis we know $\calM,w \Vdash \phi$ and $\calM,w \Vdash \psi$, hence $\calM, w \Vdash \phi \wedge \psi$.

Now assume the last inference is an instance of $L\!\imp\!\imp'$. This means that $\Gamma \Impl \Delta$ is of the form $\Gamma', (\phi \impl \psi) \impl \gamma \Impl \Delta$ and
\begin{center}
\AXC{$\Gamma', \psi \impl \gamma, \phi \Impl \psi$}
\AXC{$\Gamma', \gamma \Impl \Delta$}
\RL{$L\!\imp\!\imp'$}
\BIC{$\Gamma', (\phi \impl \psi) \impl \gamma \Impl \Delta$}
\DisplayProof
\end{center}
Let $\calM$ be a model and $w \in \calM$ such that $\calM, w \Vdash \Gamma', (\phi \impl \psi) \impl \gamma$. From $\calM, w \Vdash (\phi \impl \psi) \impl \gamma$ we can deduce that $\calM, w \Vdash \psi \impl \gamma$. The induction hypothesis of the first premise gives us $\Gamma', \psi \impl \gamma, \phi \models \psi$ and so $\calM,w \Vdash \phi \impl \psi$ by monotonicity. Since $\calM,w \Vdash (\phi \impl \psi) \impl \gamma$ we have $\calM,w \Vdash \gamma$. Hence, by the induction hypothesis applied to the second premise we have $\calM,w \Vdash \Delta$.

For $\leftimprule_{{\rm SL}}$, write $\Gamma \Impl \Delta$ as $\Pi, \Box \Gamma', \Box \phi \impl \psi \Impl \Delta$, where $\Pi$ does not contain boxed formulas. We have
\begin{center}
\AXC{$\Pi, \boxdot \Gamma',  \Box \phi, \Box \phi \impl \psi \Impl \phi$}
\AXC{$\Pi, \Box \Gamma', \psi \Impl \Delta$}
\RL{$\leftimprule_{{\rm SL}}$}
\BIC{$\Pi, \Box \Gamma', \Box \phi \impl \psi \Impl \Delta$}
\DisplayProof
\end{center} 
Suppose $\calM, w \Vdash \Pi, \Box \Gamma', \Box \phi \impl \psi$. We show that this implies $\calM,w\Vdash \psi$ and therefore $\calM,w \Vdash \Delta$ by the induction hypothesis of the second premise. Suppose for a contradiction that $\calM,w \not \Vdash \psi$. By $\calM,w \Vdash \Box \phi \impl \psi$ we have $\calM,w \not \Vdash \Box \phi$. This implies the following reasoning:
\begin{itemize}
\item there exists $x_1$ such that $w R x_1$ and $\calM,x_1 \not \Vdash \phi$. Because $\calM, w \Vdash \Box \Gamma'$, we have $\calM, x_1 \Vdash \Gamma'$. Since we work in an $\iSL$-model, we also have $w \leq x_1$ and therefore by monotonicity $\calM, x_1 \Vdash \Pi, \Box \Gamma', \Box \phi \impl \psi$. So together we have $\calM, x_1 \Vdash \Pi, \boxdot \Gamma', \Box \phi \impl \psi$ and $\calM,x_1 \not \Vdash \phi$. By induction hypothesis of the first premise, $\calM, x_1 \not \Vdash \Box \phi$.
\item So, there exists $x_2$ such that $x_1 R x_2$ and $\calM,x_2 \not \Vdash \phi$. By the above reasoning: $\calM, x_2 \Vdash \Pi, \boxdot \Gamma', \Box \phi \impl \psi$ and $\calM,x_2 \not \Vdash \phi$. So by induction hypothesis $\calM, x_2 \not \Vdash \Box \phi$.
\item So, there exists $x_3$ $\dots$
\end{itemize}
Now we can construct an infinite sequence $w R x_1 R x_2 R  x_3 R \dots$. This is in contradiction with the conversely well-foundedness of relation $R$. Hence we proved $\calM, w \Vdash \psi$, and therefore $\calM,w \Vdash \Delta$.

The last rule we treat is $\rsch_{\rm SL}^{4}$. So we have
\begin{center}
\AXC{$\Pi, \boxdot \Gamma', \Box \phi \Impl \phi$}
\RL{$\rsch_{\rm SL}^{4}$}
\UIC{$\Pi, \Box \Gamma' \Impl \Box \phi$}
\DisplayProof
\end{center}
Suppose $\calM, w \Vdash \Pi, \Box \Gamma'$. We want to prove $\calM, w \Vdash \Box \phi$. Suppose for a contradiction that $\calM, w \not \Vdash \Box \phi$. Then by a similar reasoning as for rule $\leftimprule_{{\rm SL}}$ we obtain an infinite sequence $wR x_1 R x_2 R x_3 \dots$. This is a contradiction, so $\calM, w \Vdash \Box \phi$.

\end{proof}

\subsection{Completeness} 

This section provides the countermodel construction. We use the following lemma about invertible rules. Recall that $L\!\imp\!\imp'$, $R\vee$, $\imp_{{\rm SL}}$ and $\rsch_{\rm SL}^{4}$ are the non-invertible rules for $\DYSLprime$. All other rules are invertible.

\begin{lemma}\label{lemma semantic invertible}
Let $\calM$ be a model with world $w$. For any invertible rule in $\DYSLprime$ we have that whenever one of its premises is refuted by $w$ in $\calM$, the conclusion of the rule is also refuted by $w$ in $\calM$.
\end{lemma}
\begin{proof}
This is an easy calculation for each invertible rule.
\end{proof}

\begin{theorem}[Completeness]\label{thm countermodels}
If $\Gamma \models \Delta$, then $ \proofs_\DYSLprime \Gamma \Impl \Delta$.
\end{theorem}
\begin{proof}
We prove by contradiction that whenever there is no cut-free proof for $S=(\Gamma \Impl \Delta)$, then there is an $\iSL$-model that refutes $\Gamma \Impl \Delta$. Recall that the modal relation $R$ should be transitive and conversely well-founded. We assume that $\not \proofs_\DYSLprime S$, which means that~$S$ is marked as negative in the proof search tree. We use an induction on the height $d$ of the finite proof search tree for $\Gamma \Impl \Delta$. We start with the leaves for which $d=1$. For $d>1$ we distinguish between saturated and non-saturated negative sequents. Recall that for each non-saturated sequent in the proof search tree marked as negative, there is \emph{at least} one predecessor marked as negative. And for each saturated sequent marked as negative, \emph{for all} backwards applied rules, \emph{at least} one predecessor is marked as negative. 

If $d=1$, then $\Gamma \Impl \Delta$ is a saturated sequent with no possible backwards rule. It cannot be an initial sequent since $\Gamma \Impl \Delta$ is underivable. Together with Lemma \ref{lemirreducible}, we see that formulas in $\Gamma$ have the form: $p$ with $\{p\} \neq \Delta $, $p \impl \phi$ with $p \notin \Gamma$ or $\Box \phi$. And $ \Delta =\{q\}$ with $q \notin \Gamma$, $\Delta=\{ \bot \}$ or $\Delta=\emptyset$. We consider a one-point model $\calM = (W, \leq, R, V)$ by $W = \{w \}$, $\leq = \{ (w,w)\}$, $R = \emptyset$, $V(w,p)=1$ iff $p \in \Gamma$. This is visualized by \vspace{-0.5em}
\begin{center} \begin{tikzpicture}[modal]
\node[point] (w)  [label=below:{$p$ for $p \in \Gamma$}] {};
\end{tikzpicture}
\end{center}\vspace{-0.5em}
This is clearly an $\iSL$-model with $\calM,w \Vdash \Gamma$, but $\calM,w \not \Vdash \Delta$.

Now suppose $d>1$ and $\Gamma \Impl \Delta$ is not saturated. This means that the last rule applied is invertible. Sequent $\Gamma \Impl \Delta$ is marked as negative, so there is a premise $\Gamma' \Impl \Delta'$ marked as negative. This means that $\Gamma' \Impl \Delta'$ is underivable. By induction hypothesis, there is a countermodel $\calM$ such that $\calM \not \models \Gamma' \Impl \Delta'$. By Lemma \ref{lemma semantic invertible}, we know that the rule is also semantic invertible, which means that $\calM \not \models \Gamma \Impl \Delta$.

Now let $d>1$ and suppose $\Gamma \Impl \Delta$ is saturated. This means that last rules applied in the proof search are non-invertible rules. We have different predecessor nodes of $\Gamma \Impl \Delta$, depending on the possible rules applied to $\Gamma \Impl \Delta$. For each applied rule, at least one premise is marked as negative. For each possible application of a non-invertible rule we have the following:
\begin{enumerate}[label=(\alph*)]
\item $(L\!\imp\!\imp')$: $\Gamma \Impl \Delta$ has the form $\Gamma'_i, (\phi_i \impl \psi_i) \impl \gamma_i \Impl \Delta$ with $\Gamma'_i, \psi_i \impl \gamma_i, \phi_i\Impl \psi_i$ or $\Gamma'_i, \gamma_i \Impl \Delta$ as negative predecessor in the proof search tree.
\item $(\leftimprule_{{\rm SL}})$: $\Gamma \Impl \Delta$ has the form $\Pi_j, \Box \Gamma'_j, \Box \phi_j \impl \psi_j \Impl \Delta_j$ where one of the sequents $\Pi_j, \boxdot \Gamma'_j, \Box \phi_j, \Box \phi_j \impl \psi_j \Impl \phi_j$ or $\Pi_j, \Box \Gamma'_j, \psi_j \Impl \Delta$ is a negative predecessor.
\item $(\rsch_{\rm SL}^{4})$: $\Gamma \Impl \Delta$ has the form $\Pi, \Box \Gamma' \Impl \Box \phi$ where premise $\Pi, \boxdot \Gamma', \Box \phi \Impl \phi$ is negative,
\item $(R\vee)$: $\Gamma \Impl \Delta$ has the form $\Gamma \Impl \phi \vee \psi$ and both premises $\Gamma \Impl \phi$ and $\Gamma \Impl \psi$ are negative predecessors.
\end{enumerate}

Note that case (c) and (d) cannot occur both at the same time. In the following we just say predecessor to mean negative predecessor, since those are the only of interest. We continue by analyzing two different possibilities.
\begin{itemize}
\item First assume that at least one of the following occurs: (a) with predecessor $\Gamma'_i,\gamma_i \Impl \Delta$ for some $i$ or (b) with predecessor $\Pi_j, \Box \Gamma'_j,\psi_j \Impl \Delta$ for some $j$. Those cases are treated in a similar way as the semantic invertible rules. In each case, we find a countermodel for the corresponding predecessor using the induction hypothesis. In both cases the found countermodel is also a countermodel for $\Gamma \Impl \Delta$.
\item Now assume that none of the previous cases occur. We apply the induction hypothesis to each of the applied rules of the concerned saturated sequent:
\begin{itemize}
\item In (a): for each $i$ we have predecessor $S_i = (\Gamma'_i, \psi_i \impl \gamma_i, \phi_i \Impl \psi_i)$. So by induction hypothesis there exists for each $i$ a model $\calK_i'$ and world $k_i$ such that $\calK'_i, k_i \not \Vdash S_i$. Let $\calK_i$ be the submodel with root $k_i$, then $\calK_i,k_i \not \Vdash S_i$.
\item In (b): for each $j$ we have predecessor $S_j = (\Pi_j, \boxdot \Gamma'_j, \Box \phi_j, \Box \phi_j \impl \psi_j \Impl \phi_j)$. So by induction hypothesis there exists for each $j$ a model $\calL_j'$ and world $l_j$ such that $\calL'_j, l_j \not \Vdash S_j$. Let $\calL_j$ be the submodel with root $l_j$, then $\calL_j,l_j \not \Vdash S_j$.
\item In (c): we apply the induction hypothesis to predecessor $S = (\Pi, \boxdot \Gamma', \Box \phi \Impl \phi)$. So we have a model $\calX'$ and world $x$ such that $\calX', x \not \Vdash S$. Let $\calX$ be the submodel with root $x$, then $\calX,x \not \Vdash S$.
\item In (d): apply the induction hypothesis to both predecessors $S_1 = (\Gamma \Impl \phi)$ and $S_2 = (\Gamma \Impl \psi)$ to get two models $\calH_1, \calH_2$ with roots $h_1,h_2$ such that $\calH_1,h_1 \not \Vdash S_1$ and $\calH_2, h_2 \not \Vdash S_2$. 
\end{itemize}

Sequent $\Gamma \Impl \Delta$ is saturated, so $\Delta$ is a disjunction, boxed formula, an atom, bottom or empty. The latter two are equivalent. When $\Delta$ contains a boxed formula, say $\Delta=\{ \Box \phi \}$, case (c) occurs, but case (d) does not. Then we construct the following model $\calM$. In this picture, relation $R$ should be understood as the transitive closure and whenever $v \leq y R z$ for some $v,y,z$ then $v R z$. Dashed arrows represent relation $\leq$ and the other arrows represent relation $R$. Double headed arrows are examples of $R$ relations defined by the closure conditions.

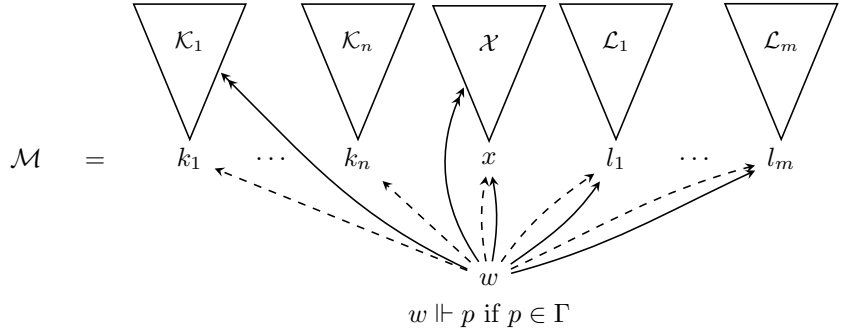
\begin{center} \begin{tikzpicture}[modal]
\node[world] (w) [label=below:{$w \Vdash p$ if $p \in \Gamma$}] {$w$};
\node[world] (x) [above =of w] {$x$};
\node[itria] (X) [above =of x, yshift=2.5mm]{\small{$\calX$}};
\node[world] (kn) [left =of x] {$k_n$};
\node[itria] (Kn) [above =of kn, yshift=2mm]{\small{$\calK_n$}};
\node[world] (ki) [left =of kn, xshift=0.75cm] {$\dots$};
\node[world] (k1) [left =of ki, xshift=0.75cm] {$k_1$};
\node[itria] (K1) [above =of k1, yshift=2mm]{\small{$\calK_1$}};
\node[world] (l1) [right =of x] {$l_1$};
\node[itria] (L1) [above =of l1, yshift=2mm]{\small{$\calL_1$}};
\node[world] (lj) [right =of l1, xshift=-0.75cm] {$\dots$};
\node[world] (lm) [right =of lj, xshift=-0.75cm] {$l_m$};
\node[itria] (Lm) [above =of lm, yshift=2mm]{\small{$\calL_m$}};
\node[world] (model) [left =of k1, xshift = 0.5cm] {$\calM \ \ \ \ = $};

\path[->] (w) edge[dashed]   (k1);
\path[->] (w) edge[dashed]  (kn);
\path[->] (w) edge[out=100, in=260, dashed]  (x);
\path[->] (w) edge[out=80, in=280]  (x);
\path[->] (w) edge[out=55, in=215, dashed]  (l1);
\path[->] (w) edge[out=35, in=235]  (l1);
\path[->] (w) edge[out=25 ,in=195, dashed]  (lm);
\path[->] (w) edge[out=15, in=205]  (lm);
\path[->>] (w) edge[bend left]  (X);
\path[->>] (w) edge[out=155,in=-45]  (K1);
\end{tikzpicture}
\end{center}

For each $y,v,z$ we have $y\leq v Rz$  implies $yRv$ and $yRv$ implies $w \leq v$ by construction. Relation $R$ is transitive and conversely well-founded by construction. Further note that the valuation in $\calM$ respects monotonicity, since $\calM,k_i \Vdash p$, $\calM,x \Vdash p$ and $\calM,l_j \Vdash p$ for each $p\in \Gamma$. This can be seen by inspection of the rules. So we can conclude that~$\calM$ is an $\iSL$-model.

Now we claim that this model refutes sequent $\Gamma \Impl \Delta$, where $\Delta=\{ \Box \phi \}$. We show that $\calM,w \Vdash \psi$ for each $\psi \in \Gamma$ and $\calM, w \not \Vdash \Box \phi$. The latter follows directly by the induction assumption in (c) implying $\calM,x \not \Vdash \phi$. For the first, recall that sequent $\Gamma \Impl \Box \phi$ is a saturated sequent. We treat each $\psi \in \Gamma$ depending on its form.
\begin{itemize}
\item For $\Box \psi \in \Gamma$, let $y$ such that $w R y$. We have to show that $\calM,y\Vdash \psi$. If $y \in \calK_i$, then $\calM,y\Vdash \psi$ since $k_i R y$ and $\calM,k_i \Vdash \Box \psi$. Now consider $y \in \calX$. By construction of $R$, we find a finite sequence $y=z_1, z_2,  \dots, z_{n-1},z_n = x$ such that $xRz_{n-1}R\dots Rz_2Ry$. So $xRy$, because $R$ is transitive. By induction assumption (c) we have $\calM,x \Vdash \Box \psi,\psi$. Hence $\calM, y \Vdash \psi$. For $y \in \calL_j$ we proceed by a similar argument.

\item For $p \in \Gamma$, then $\calM,w \Vdash p$ by definition,

\item For $p \impl \psi \in \Gamma$, then $p \notin \Gamma$ because $\Gamma \Impl \Delta$ is saturated. Let $v \geq w$ and $\calM,v \Vdash p$. Since $p \notin \Gamma$ we have $v > w$ and therefore $v \geq k_i$ for some $i$ or $v \geq x$ or $v \geq l_j$ for some $j$. For each case we have that $p \impl \psi$ is also present in the antecedent of the predecessor. So $\calM, k_i \Vdash p \impl \psi$ for each $i$, $\calM, x \Vdash p \impl \psi$ and $\calM, l_j \Vdash p \impl \psi$ for each $j$. So indeed $\calM, v \Vdash \psi$.

\item For $(\phi_{i'} \impl \psi_{i'}) \impl \gamma_{i'} \in \Gamma$ for some $i'$. Let $v \geq w$ such that $\calM, v \Vdash \phi_{i'} \impl \psi_{i'}$. We treat different cases for $v$. For $v=w$, we use induction in (a) saying that $\calM, k_{i'} \not \Vdash \psi_{i'}$ and $\calM, k_{i'} \Vdash \phi_{i'}$. This implies $\calM,k_{i'} \not \Vdash \phi_{i'} \impl \psi_{i'}$. By monotonicity, $\calM,w \not \Vdash \phi_{i'} \impl \psi_{i'}$ which contradicts our assumption. So $v > w$ and again we have several possibilities: $v \geq k_{i'}$, $v \geq k_{i}$ for some $i \neq i'$, $v \geq x$ or $v \geq l_j$ for some $j$. For $v \geq k_{i'}$ we use $\calM, k_{i'} \Vdash \psi_{i'} \impl \gamma_{i'}$, $\calM, k_{i'} \Vdash \phi_{i'}$ and $\calM, k_{i'} \not \Vdash \psi_{i'}$. By monotonicity, $\calM, v \Vdash \psi_{i'} \impl \gamma_{i'}$ and $\calM, v \Vdash \phi_{i'}$. Using assumption $\calM, v \Vdash \phi_{i'} \impl \psi_{i'}$ gives $\calM, v \Vdash \gamma_{i'}$. For all other cases of $v$ we can use the fact that $(\phi_{i'} \impl \psi_{i'}) \impl \gamma_{i'}$ stays present in the predecessor of the rule. So we can conclude $\calM, k_i \Vdash (\phi_{i'} \impl \psi_{i'}) \impl \gamma_{i'}$ for each $i \neq i'$, $\calM, x \Vdash (\phi_{i'} \impl \psi_{i'}) \impl \gamma_{i'}$ and $\calM, l_j \Vdash (\phi_{i'} \impl \psi_{i'}) \impl \gamma_{i'}$ for each $j$. This implies $\calM, v \Vdash \gamma_{i'}$.

\item For $\Box \phi_{j'} \impl \psi_{j'} \in \Gamma$ for some $j'$, let $v \geq w$ and $v \Vdash \Box \phi_{j'}$. We see that $v \neq w$, because suppose $v = w$, we have by induction in (b) $\calM, l_{j'} \not \Vdash \phi_{j'}$, hence $w \not \Vdash \Box \phi_{j'}$ which contradicts our assumption. So $v \geq k_i$ for some $i$, $v \geq x$ or $v \geq l_j$ for some $j$. For all those cases we have $\Box \phi_{j'} \impl \psi_{j'}$. This is also true for $v \geq l_{j'}$ by keeping $\Box \phi_{j'} \impl \psi_{j'}$ in the premise of rule $\leftimprule_{{\rm SL},1}$.
\end{itemize}
To conclude, we found an $\iSL$-model $\calM$ such that $\calM \not \models \Gamma \Impl \Delta$, where $\Delta$ is a boxed formula $\Box \phi$. 

When $\Delta$ is a disjunction $\phi \vee \psi$, case (d) occurs, but case (c) does not. In this case we construct the following countermodel $\calM'$, defined in a similar way as before.
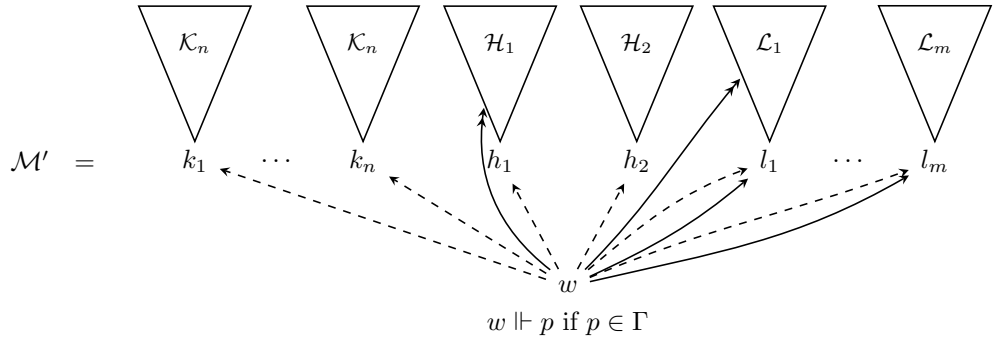
\begin{center} \begin{tikzpicture}[modal]
\node[world] (w) [label=below:{$w \Vdash p$ if $p \in \Gamma$}] {$w$};
\node[world] (h1) [above =of w, xshift=-0.9cm] {$h_1$};
\node[itria] (H1) [above =of h1, yshift=2mm]{\small{$\calH_1$}};
\node[world] (h2) [right =of h1] {$h_2$};
\node[itria] (H2) [above =of h2, yshift=2mm]{\small{$\calH_2$}};
\node[world] (kn) [left =of h1] {$k_n$};
\node[itria] (Kn) [above =of kn, yshift=2mm]{\small{$\calK_n$}};
\node[world] (ki) [left =of kn, xshift=0.75cm] {$\dots$};
\node[world] (k1) [left =of ki, xshift=0.75cm] {$k_1$};
\node[itria] (K1) [above =of k1, yshift=2mm]{\small{$\calK_n$}};
\node[world] (l1) [right =of h2] {$l_1$};
\node[itria] (L1) [above =of l1, yshift=2mm]{\small{$\calL_1$}};
\node[world] (lj) [right =of l1, xshift=-0.75cm] {$\dots$};
\node[world] (lm) [right =of lj, xshift=-0.75cm] {$l_m$};
\node[itria] (Lm) [above =of lm, yshift=2mm]{\small{$\calL_m$}};
\node[world] (model) [left =of k1, xshift = 0.3cm] {$\calM' \ \  = $};

\path[->] (w) edge[dashed]   (k1);
\path[->] (w) edge[dashed] (kn);
\path[->] (w) edge[dashed] (h1);
\path[->] (w) edge[dashed]  (h2);
\path[->] (w) edge[dashed, out=40, in=200]  (l1);
\path[->] (w) edge[out=25, in=220]  (l1);
\path[->] (w) edge[dashed, out=23, in = 200] (lm);
\path[->] (w) edge[out=13, in=210] (lm);
\path[->>] (w) edge[bend left]  (H1);
\path[->>] (w) edge[out=45,in=230, yshift=2mm] (L1);
\end{tikzpicture}
\end{center}
The proof is similar as before. When $\Delta$ is an atom $p$ or $\bot$ we only create the model using $\calK_i$ and $\calL_j$, because both case (c) and (d) do not occur.
\end{itemize}
\end{proof}

Note that the proof of Theorem \ref{thm countermodels} shows us that $\iSL$ is complete with respect to finite tree-like models. 

\begin{corollary}\label{cor:cutSL4a}
$\DYSLprime$ admits the rules {\it Contraction} and {\it Cut}.
\end{corollary}

\section{Other axiomatizations}
 \label{secother}
In this section we discuss alternative axiomatizations of $\iSL$. \

\subsection{Hilbert systems}
The Hilbert systems for $\iSL$ are already briefly discussed in the introduction. Here we state them again to provide a complete list of proof systems for $\iSL$. Let $\iHSL$ denote the Hilbert system consisting of the Hilbert system $\iK$ extended by the Strong L\"ob Principle 
\[
 {\rm (SL)} \ \ \  (\bx\phi \imp \phi) \imp \phi.
\]
Here $\iK$ consists of a standard Hilbert system for $\IPC$ extended by the axiom 
\[
 {\rm (K)} \ \ \ \bx(\phi\imp \psi) \imp (\bx\phi\imp\bx\psi)
\]
and the Necessitation rule (from $\proofs \phi$ infer $\proofs \Box \phi$). Cut-elimination for $\LJSL$ (Theorem~\ref{thmcutadm}), implies
\[
 \af_{\iHSL} \phi \text{ if and only if } \af_{\LJSL}(\ \seq \phi). 
\]
Another Hilbert style axiomatization, denoted $\iHGLC$, consists of $\iK$, L\"ob's Principle {\rm (GL)}  and the Completeness Principle {\rm (CP)} \citep{litak14}.
\[
  {\rm (GL)} \ \ \  \bx(\bx\phi \imp \phi) \imp \bx\phi \ \ \ \ \ \ \ 
  {\rm (CP)} \ \ \  \phi \imp \bx\phi.
\] 

\subsection{Sequent calculi}
We have seen two variants of the modal rule, that is $\rsch_{\rm SL}$ and $\rsch_{\rm SL}^4$, and we can also identify a third one that we call $\rsch_{\rm SL}^{4a}$: 
$$
  \infer[\rsch_{\rm SL}]{\bx\Sig,\Pi,\bx\Ga \seq \bx\phi}{\Pi,\dbx\Ga,\bx\phi \seq \phi} \ \ \ \ 
  \infer[\rsch_{\rm SL}^4]{\Pi,\bx\Ga \seq \bx\phi}{\Pi,\dbx\Ga,\bx\phi \seq \phi}  \ \ \ \ 
  \infer[\rsch_{\rm SL}^{4a}]{\Sigma,\bx\Ga \seq \bx\phi}{\Sigma,\dbx\Ga,\bx\phi \seq \phi}
$$
There is no restriction on the formulas in $\Sig$, but $\Pi$ does not contain boxed formulas. Based on these rules we can define different but equivalent systems in addition to the systems defined in Figure~\ref{figlogics}:

\[
 \begin{array}{ll} 
\text{name} & \text{calculus}  \\
\LJSLvier\  & \LJm +\rsch_{\rm SL}^{4}  \\
\LJSLa\  & \LJm + \rsch_{\rm SL}^{4a}  \\
\DYSLa\ & \DYm + \leftimprule_{\rm SL} + \rsch_{\rm SL}^{4a} \\
 \end{array}
\]

\begin{lemma}\label{lem systems equivalent}
Systems $\LJSL$, $\LJSLvier$, $\LJSLa$, $\DYSL$ and $\DYSLa$ are equivalent. In addition, $\DYSL$ and $\DYSLa$ are terminating modulo extended axioms.
\end{lemma}
\begin{proof}
By Theorem~\ref{thmequivalence}, systems $\LJSL$ and $\DYSL$ are equivalent. Therefore, it is sufficient to show that $\LJSLvier$ and $\LJSLa$ are equivalent to $\LJSL$ and that $\DYSLa$ is equivalent to $\DYSL$. For each equivalence we have to show two directions which are all proved by induction on the height $d$ of a derivation. Since all axioms and nonmodal rules of the considered systems are equal, it suffices to focus on the modal rules. We only rely on the height-preserving weakening in $\LJSL$ and $\DYSL$ from Lemma~\ref{lemstructural}. Let us only look at the equivalence between $\LJSL$ and $\LJSLa$. For one direction suppose that the last inference of the derivation in $\LJSL$ looks as follows:
\[
 \infer[\rsch_{\rm SL}]{\Box \Sig, \Pi, \bx\Ga \seq \bx\phi}{ \Pi, \dbx\Ga,\bx\phi \seq \phi}
\]
By height-preserving weakening $\proofs^{d-1}_\LJSL (\Box \Sigma, \Pi, \boxdot \Gamma, \Box \phi \Impl \phi)$. By the induction hypothesis this is also derivable in $\LJSLa$. An application of $\rsch_{\rm SL}^{4a}$ gives the desired result.
For the other direction suppose the last inference of the derivation looks as follows:
\[
 \infer[\rsch_{\rm SL}^{4a}]{\Sig,\bx\Ga \seq \bx\phi}{\Sig,\dbx\Ga,\bx\phi \seq \phi}
\]
By the induction hypothesis there exists a proof in $\LJSL$ of the premise. Suppose that $\Sig=\Sig^n \cup \bx\Sig^b$, where the formulas in $\Sig^n$ are not boxed. By Lemma~\ref{lemstructural} there exists a proof in $\LJSL$ of $(\Sig^n,\Sig^b,\bx\Sig^b,\dbx\Ga,\bx\phi \seq \phi)$ as well. An application of $\rsch_{\rm SL}$ shows that the endsequent has a proof in $\LJSL$. 

Termination modulo extended axioms for $\DYSL$ is shown in Theorem \ref{thmtermination} and the same proof applies to $\DYSLa$.
\end{proof}

Hilbert style calculus $\iHGLC$ suggests a sequent calculus for $\iSL$ consisting of $\LJGL$ and the following rule:
\[
 \infer[L\Box]{\Ga, \phi \seq \De}{\Ga,\bx\phi \seq \De}
\]

We call it $\LJGLC$. 

\begin{lemma}
$\LJGLC$ and $\LJSL$ are equivalent.
\end{lemma}
\begin{proof}
We prove that the rules of $\LJGLC$ are derivable in $\LJSL$ and vice versa. Each nonmodal rule is the same for both calculi, so we only consider the modal rules. To prove that $\LJGLC$ implies $\LJSL$, we show that $\rsch_{\rm SL}$ is derivable in $\LJGLC$. It is easy to see that $\LJGLC$ admits weakening.
\begin{center}
\AXC{$\Pi, \boxdot \Gamma, \Box \phi \Impl \phi$}
\doubleLine
\RL{{\it Weakening}}
\UIC{$\boxdot \Pi, \boxdot \Gamma, \Box \phi \Impl \phi$}
\RL{$\rsch_{\rm GL}$}
\UIC{$\Box \Sigma, \Box \Pi, \Box \Gamma \Impl \Box \phi$}
\RL{$L\Box$}
\UIC{$\Box \Sigma, \Pi, \Box \Gamma \Impl \Box \phi$}
\DisplayProof
\end{center}
For the other direction we have to show that $\rsch_{\rm GL}$ and $L\Box$ are derivable in $\LJSL$. $\rsch_{\rm GL}$ is easy and the argument for $L\Box$ is as follows, where we use the cut rule in $\LJSL$. We consider the two cases, where $\phi$ is boxed (say $\phi = \Box \phi'$) or not.
\begin{center}
\AXC{$\boxdot \phi', \Box \Box \phi' \Impl \Box \phi'$}
\RL{$\rsch_{\rm SL}$}
\UIC{$\Box \phi' \Impl \Box \Box \phi'$}
\AXC{$\Gamma, \Box \Box \phi' \Impl \Delta$}
\RL{$Cut$}
\BIC{$\Gamma, \Box \phi' \Impl \Delta$}
\DisplayProof  \ \ \ 
\AXC{$\phi, \Box \phi \Impl \phi$}
\RL{$\rsch_{\rm SL}$}
\UIC{$\phi \Impl \Box \phi$}
\AXC{$\Gamma, \Box \phi \Impl \Delta$}
\RL{$Cut$}
\BIC{$\Gamma, \phi \Impl \Delta$}
\DisplayProof
\end{center} 
\end{proof}

We say that a Hilbert system is equivalent to a sequent calculus if whenever formula $\phi$ is provable in the Hilbert system, the sequent $(\Impl \phi)$ is provable in the sequent calculus. The following equivalences have been obtained in this paper.  

\begin{corollary}[Equivalence proof systems $\iSL$] \\
Proof systems  $\LJSL$, $\LJSLa$, $\LJSLvier$,  $\DYSL$, $\DYSLprime$, $\DYSLa$, $\LJGLC$, $\iHSL$ and $\iHGLC$  are all equivalent. 
\end{corollary}
\begin{proof}
We sum up the equivalences provided in the paper: $\LJSL$ is equivalent to $\LJSLa$, $\LJSLvier$,  $\DYSL$, $\DYSLa$, $\LJSLvier$ and $\LJGLC$ by the previous two lemmas. 
Calculus $\DYSLprime$ is shown to admit the cut rule (Corollary \ref{cor:cutSL4a}), which implies equivalence to $\LJSL$. We know that all sequent calculi admit the cut rule. For each sequent calculus this implies equivalence to Hilbert system $\iHSL$. Equivalence of~$\iHGLC$ with~$\DYSLprime$ follows from completeness of Kripke semantics and \citep{litak14}, where completeness for $\iHGLC$ with respect to $\iSL$-models is stated.
\end{proof}

We conclude this section with the following thought. One might expect that the following rules are also sound and complete for $\iSL$, where we drop each $\boxdot$, since $\boxdot \chi \leftrightarrow \chi$ for each formula $\chi$ by the Completeness Principle.
\[ \begin{array}{cc} 
  \infer[\rsch_{\rm SL}']{\bx\Sig,\Pi,\bx\Ga \seq \bx\phi}{\Pi,\Ga,\bx\phi \seq \phi} & 
  \infer[\rsch^{4'}_{\rm SL}]{\Pi,\bx\Ga \seq \bx\phi}{\Pi,\Ga,\bx\phi \seq \phi}  \\
  \\
  \infer[\rsch^{4a'}_{\rm SL}]{\Sigma,\bx\Ga \seq \bx\phi}{\Sigma,\Ga,\bx\phi \seq \phi}  &
  \infer[\leftimprule_{\rm SL}']{\Pi,\bx\Ga,\bx\phi \imp \psi \seq \De}{
   \Pi,\Ga,\bx\phi, \Box \phi \impl \psi \seq \phi & \Pi,\bx\Ga,\psi \seq \De}

\end{array}
\]
It is easily shown that the rules are sound, but completeness is again difficult (or impossible) to prove. The current order to show termination with rule $\rsch^{4}_{\rm SL}$ relies on the fact that boxed
formulas do not disappear in the antecedents by bottom-up applications of the rules. This crucial fact is not true for $\rsch^{4'}_{\rm SL}$. As formula $\phi$ is ‘duplicated’ in the premise, the question is with what order the system with this rule would terminate.

\section{Conclusion}
This paper introduces several sequent calculi for intuitionistic strong L\"ob logic $\iSL$, the two main ones are 
$\LJSL$ and $\DYSL$.
$\LJSL$ and $\DYSL$ are extensions of $\LJ$ and $\DY$, respectively, and it is shown that in both calculi the structural rules {\it Weakening}, {\it Contraction} and {\it Cut} are admissible. For $\DYSL$ it does not suffice to add a modal rule to $\DY$, as is the case for $\LJSL$, also an implication rule has to be added in order to deal with implications on the left which antecedent is a boxed formula. 

It is shown in Sections~\ref{sectermination} and \ref{secinterpolation} that $\DYSL$ is a terminating proof system and that it has interpolation (the termination is not needed for that), thus implying that $\iSL$ has interpolation, as is known from the literature. 
A constructive proof of the equivalence of~$\LJSL$ and~$\DYSL$ is provided in Section~\ref{secequivalence}. As mentioned in the introduction, terminating systems for a small class of intuitionistic modal logics are used to establish results on uniform interpolation \citep{iemhoff17}. In contrast to this work, the strong termination of $\DYSL$ is modulo extended axioms and is based on another well-ordering. It raises the question whether the methods developed in \citep{iemhoff17} can be adjusted to provide a syntactic proof of uniform interpolation for $\iSL$.

A well-known semantics for intuitionistic modal logic is the class of Kripke models with two relations, a modal relation  $R$ and a partial order $\leq$, the intuitionistic relation. As usual, the models are upwards persistent along the intuitionistic relation $\leq$ and $w \leq v R x$ implies $wRx$ for all nodes $w,v,x$. As is known from the literature \citep{litak14}, $\iSL$ is sound and complete with respect to the class of Kripke models in which $R$ is transitive, conversely well-founded and a subset of $\leq$. Here we give a proof of this fact by showing that a variant of $\DYSL$, $\DYSLprime$, is sound and complete with respect to that same class of models. For this we use a countermodel construction that given an underivable sequent produces a finite countermodel, thus resulting in a constructive proof of completeness for $\iSL$.

\paragraph{Acknowledgements.}
We thank Tadeusz Litak and Albert Visser for several enjoyable conversations about Intuitionistic Strong L\"ob Logic and related matters. We further thank the reviewers for their useful remarks for improvement. Support by the Netherlands Organisation for Scientific Research under grant 639.073.807 is gratefully acknowledged.

\end{document}